\documentclass[10pt]{article}
\usepackage{amsbsy}
\usepackage{amsmath,amsfonts,amssymb,amsthm,mathrsfs,mathtools,stmaryrd}
\usepackage[a4paper,vmargin={3.5cm,3.5cm},hmargin={2.5cm,2.5cm}]{geometry}
\usepackage[font=sf, labelfont={sf,bf}, margin=1cm]{caption}
\usepackage{graphicx,graphics,xcolor}
\usepackage{epsfig}%pour les eps
\usepackage{latexsym}%encore des symboles
\usepackage[applemac]{inputenc}
\usepackage{ae,aecompl}
\usepackage{soul}
\usepackage[english]{babel}
\usepackage[pdfpagemode=UseNone,bookmarksopen=false,colorlinks=true,urlcolor=blue,citecolor=blue,citebordercolor=blue,linkcolor=blue]{hyperref}
\usepackage{tikz,mathtools}%,pstricks}
\usepackage{enumerate}

\usepackage{bbm,tikz-cd,ulem}
\makeatletter
\newcommand{\fixed@sra}{$\vrule height 2\fontdimen22\textfont2 width 0pt\shortuparrow$}
\newcommand{\newsearrow}{%
  {\!\mathrel{\text{\rotatebox[origin=c]{\numexpr 5*45}{\fixed@sra}}\!}}
}
\makeatother

\renewcommand{\P}{\mathbb{P}}

\def\A{\mathtt{A}}

\def\U{\mathbb{U}}
\def\N{\mathbb{N}}
\def\X{\mathbb{X}}
\def\x{\mathbf{x}}
\def\e{{\rm e}}
\def\d{{\rm d}}
\newcommand{\Q}{\mathbb{Q}}
\newcommand{\R}{\mathbb{R}}

\newtheorem{theorem}{Theorem}[section]

\newtheorem{proposition}[theorem]{Proposition}
\newtheorem{lemma}[theorem]{Lemma}
\newtheorem{corollary}[theorem]{Corollary}

\theoremstyle{definition}
\newtheorem{remark}[theorem]{Remark}

\def\f{\mathcal{F}}

\def\build#1_#2^#3{\mathrel{
\mathop{\kern 0pt#1}\limits_{#2}^{#3}}}

\title{\vspace{-2cm}\textsc{On conditioning a  self-similar \\ growth-fragmentation 
by its  intrinsic area } \vspace{0.5cm}}

\date{}

\DeclareSymbolFont{extraup}{U}{zavm}{v}{n}
\DeclareMathSymbol{\varheart}{\mathalpha}{extraup}{86}
\DeclareMathSymbol{\vardiamond}{\mathalpha}{extraup}{87}

\makeatletter
\renewcommand*{\@fnsymbol}[1]{\ensuremath{\ifcase#1\or \spadesuit \or \varheart \or \vardiamond \or \clubsuit\or \or
   \mathsection\or \mathparagraph\or \|\or **\or \dagger\dagger
   \or \ddagger\ddagger \else\@ctrerr\fi}}
\makeatother

\author{Jean Bertoin \thanks{Universit\"at Z\"urich.\hfill  \texttt{jean.bertoin@math.uzh.ch}} \qquad  Nicolas Curien \thanks{Universit\'e Paris-Sud.\hfill  \texttt{nicolas.curien@gmail.com}}  \qquad Igor Kortchemski \thanks{CNRS \& CMAP, \'Ecole polytechnique. \hfill \texttt{igor.kortchemski@math.cnrs.fr}}}

\date{}

\DeclareSymbolFont{extraup}{U}{zavm}{v}{n}
\DeclareMathSymbol{\varheart}{\mathalpha}{extraup}{86}
\DeclareMathSymbol{\vardiamond}{\mathalpha}{extraup}{87}

\makeatletter
\renewcommand*{\@fnsymbol}[1]{\ensuremath{\ifcase#1\or \spadesuit \or \varheart \or \vardiamond \or \clubsuit\or \or
   \mathsection\or \mathparagraph\or \|\or **\or \dagger\dagger
   \or \ddagger\ddagger \else\@ctrerr\fi}}
\makeatother

\begin{document}
\normalem
\maketitle

\begin{abstract}  
The genealogical structure of self-similar growth-fragmentations can be described in terms of a branching random walk. The so-called intrinsic area $\A$  arises  in this setting as the terminal value of a remarkable additive martingale. Motivated by connections with some models of random planar geometry, the purpose of this work is to investigate the effect of conditioning a self-similar growth-fragmentation on its intrinsic area. The distribution of $\A$ satisfies a useful smoothing transform which  enables us to establish the existence of a regular density $a$ and to determine the asymptotic behavior of $a(r)$ as $r\to \infty$ (this can be seen as a local version of Kesten-Grincevi\v cius-Goldie theorem's for random affine fixed point equations in a particular setting).  In turn, this yields a family of martingales from which the formal conditioning on $\A=r$ can be realized by probability tilting. We point at a limit theorem for
the conditional distribution given $\A=r$ as $r\to \infty$, and also observe that such  conditioning still makes sense under the so-called canonical measure for which the growth-fragmentation starts from $0$. 

 \end{abstract}

\medskip


\medskip


\section{Introduction}

A Markovian growth-fragmentation is a Crump-Mode-Jagers type branching process which 
can be thought of as a model describing masses of individuals in a family of living cells. These evolve independently one from the other, and the dynamics of the mass of a typical cell are governed by a Markov process on $\R_+$. Each  negative jump-time for the mass is interpreted as a birth event, in the sense that it is the time at which a daughter cell is born, whose initial mass is precisely given by the absolute size of the jump (so that conservation of masses holds at birth events). 
When those dynamics further enjoy self-similarity, the process that records masses  of cells at birth given their generations is a branching random walk. Under fairly general assumptions, this naturally yields a remarkable martingale, whose terminal value $\A$ is called the intrinsic area of the growth-fragmentation, see Section 2 in \cite{BBCK}.
The intrinsic area $\A$ is a fundamental random variable which notably appears in a variety of limit theorems for self-similar growth-fragmentations, see Dadoun \cite{Dadoun}; we also refer to the well-known contributions  \cite{Jagers, JN84, Nerman}
 by Jagers and Nerman  for closely related works where akin intrinsic martingales now are determined by the so-called Malthusian parameter.

\paragraph{Motivations.} The purpose of the present work is to investigate the conditional versions of such self-similar growth-fragmentation processes given $\A$. One of our motivations comes from recent connections \cite{BBCK, BCK, LGR, MillerS} between growth-fragmentations and random planar geometry.  In these connections,  growth-fragmentations with a given intrinsic area correspond, intuitively speaking,  to random surfaces with a given ``area". Indeed, in \cite{BCK}, we have considered cycles obtained by slicing at all heights random Boltzmann triangulations with a simple boundary, and we have established a functional invariance principle for the lengths of these cycles,
appropriately rescaled, as the size of the boundary grows. The limiting process is described using a self-similar growth-fragmentation process with explicit parameters, and roughly speaking encodes the breadth-first search of the Brownian disk (as considered in \cite{MillerS}), and the intrinsic area of the growth-fragmentations corresponds to the ``area" of the Brownian disk. See \cite[Section 6]{BBCK} for the more general case of stable Boltzmann planar maps.

\paragraph{Regularity of the density of the intrinsic area.} Of course, the study of conditioning a growth-fragmentation on the value of its intrinsic area $\A$ requires first investigating the distribution of $\A$.
In the setting of branching random walks, distributions of terminal values of  Biggins' additive martingales are usually not known explicitly, and there is a vast literature about their properties that is based on the so-called smoothing transform. 
We refer notably the treatise \cite{BDM} and the recent works \cite{DaMent, Ment} in which many more references can be found. 

We shall consider the framework where the dynamics describing masses of individuals is given by a non-negative self-similar Markov process with no positive jumps,  and we denote by $\kappa: \R_+ \to (-\infty, \infty]$ the so-called associated cumulant function (see \eqref{eqkappa} for a precise definition), which is a convex function with $\lim_{q\to\infty} \kappa(q) =\infty$. Throughout this work we make the fundamental assumption that
Cram\'er's condition is in force:
$$\text{ there exist $0<\omega_-<\omega_+<\infty$ such that $\kappa(\omega_{\pm})=0$ and $\kappa'(\omega_-)>-\infty$.
 }
$$
Also, as in \cite{BBCK} we suppose that $\kappa$ in finite in a neighborhood of $\omega_+$. Specifying to our case
general results due to Liu \cite{Liu1, Liu2}, we shall first establish the existence of a regular density $a$ for the intrinsic area $\A$ under ${\mathcal P}_1$, which is the law of the system starting from one particle of size $1$ (see Section \ref{sec:notation} for precise definitions, and in particular \eqref{E:defA} for the definition of $\A$). 

\begin{theorem}\label{thm:intro1}  The law of $\A$ 
under ${\mathcal P}_1$  is absolutely continuous. More precisely, there exists  $a\in {\mathcal C}^{\infty}_0(\R_+^*)$ (i.e. for any $n\geq 0$,  the derivative of order $n$, $a^{(n)}$, is a continuous function on $(0,\infty)$ which vanishes both at $0$ and at $\infty$)  such that
$${\mathcal P}_1(\A\in\d r) = a(r) \d r.$$
\end{theorem}

We  then establish precise local estimates for the asymptotic behavior of $a(r)$ as $r\to \infty$ and we  also actually see that $a$ is actually everywhere positive on $(0,\infty)$.  We recall first from Lemma 2.3 in \cite{BBCK} the following consequence of a general estimate also due to Liu \cite{Liu1} (related to the famous  Kesten-Grincevi\v cius-Goldie theorem) for the tail distribution:
\begin{equation}\label{E:tailA}
{\mathcal P}_1(\A > r)  \quad \mathop{\sim}_{r\to \infty} \quad  cr^{-\omega_+/\omega_-},
\end{equation}
where $c\in(0,\infty)$ is some constant. 
The following result is a local and sharper version of \eqref{E:tailA}.
\clearpage

\begin{theorem}
\label{thm:intro2} The following assertions hold.
\begin{enumerate}
\item[(i)] We have $$a(r) \quad \mathop{\sim}_{r\to \infty} \quad c\frac{\omega_+}{\omega_-}  r^{ -1-\omega_+/\omega_-},$$
where $c$ is the constant appearing in \eqref{E:tailA}.
\item[(ii)] For every $r>0$, we have $a(r)>0$.
 \end{enumerate} 
\end{theorem}

The proof of Theorem \ref{thm:intro2} occupies a major part in this work. It will be achieved through the use of several recursive distributional equations related to the smoothing transform satisfied by $ \A$. Those fixed point equations are derived via first-passage times and path decompositions of L\'evy processes with no positive jumps, since the latter arise naturally in the description of the trajectories of cells via the well-known Lamperti's transformation for self-similar Markov processes; see Lemma \ref{LB}. The absence of positive jumps is a crucial assumption in several arguments of the proof of Theorem \ref{thm:intro2} but not elsewhere (in particular Theorem \ref{thm:intro1} holds for growth-fragmentations with  positive jumps).  However, Theorem \ref{thm:intro2} should hold in a greater generality, see in particular the growth-fragmentations with positive jumps considered in  Section 6 of \cite{BBCK} for which $ \A$ is a biased stable law. We refer to Section \ref{sec:discussion} for possible extensions. 

\paragraph{Conditioning on the intrinsic area.} Properties of the density $a$ are essential to 
construct a regular disintegration of self-similar growth-fragmentations with respect to their intrinsic areas (Corollary \ref{C6}). More precisely we shall identify  martingales given in terms of the density $a$ (Theorem \ref{T2}). This allows us to define a new probability distribution by tilting (see \eqref{E:Doob}), which roughly speaking amounts on conditioning on having a fixed intrinsic area. We shall then use this representation to study the asymptotic behavior of the conditional distributions,
first given $\A=r\to \infty$ (Corollary \ref{C5}), and then when the growth-fragmentation has initial mass zero (Lemma \ref{L1}), which requires
working under the so-called canonical measure.  
In the connection with random planar geometry, the first  setting amounts to working with certain marked random surfaces without a boundary. Indeed, a particular realization of the infinite measure ${\mathcal N}^-_0$ defined in Section \ref{ssec:condarea} corresponds to the so-called Boltzmann
measure on the space of Brownian map instances in \cite{MillerS}. It should also appear in the extension of the functional limit theorem of \cite{BCK} to Boltzmann triangulations without a boundary. This will be investigated in future work.

We stress that the construction of growth-fragmentations conditioned on its intrinsic area only uses Theorem \ref{thm:intro2} as an input and works mutatis mutandis in the presence of positive jumps. In particular, it applies to the processes considered in \cite[Section 6]{BBCK} related to stable Boltzmann maps.

 \paragraph{Outline.} The plan of the rest of this article is as follows. General notation and background on growth-fragmentations and cell systems are presented in Section \ref{sec:notation}. Section \ref{sec:area} is devoted to properties of the density of the intrinsic area, and the applications to conditioning on $\A=r$ are developed in Section \ref{sec:conditioning}.

\section{Notation, assumptions, and background}
\label{sec:notation}

We lift from \cite{BBCK} some notation and assumptions related to self-similar Markov processes, cell systems, growth-fragmentation processes,  etc.  and several related notions that will be used throughout this text. As we shall need to work with many different laws or measures, it will be convenient to adopt canonical notation, in the sense that we shall denote by $X, {\mathcal X}, \X$, etc. coordinate processes on some specific spaces of functions, which are then endowed with different probability (or even $\sigma$-finite) measures, $P,Q,{\mathcal P}, {\mathcal Q}, \P, \Q$, etc. As a consequence, different notation for mathematical expectations such as ${\mathcal E}$ and $E$ can be used, sometimes in the same formula. 

\paragraph{$\bullet$ Generic trajectories and rescaling.} 
In this work, we consider c\`adl\`ag functions $w: [0,\infty)\to [0,\infty)$ that are stopped at their first hitting time of $0$. That is, if we write $\zeta_w\coloneqq \inf\{t>0: w(t)=0
\}$, then $w(s)= 0$ for all $s\geq \zeta_w$. We call $\zeta_w$ the lifetime of $w$; we  stress that $\zeta_w$ may be infinite (i.e. $w(s)>0$ for all $s>0$) and also that 
$w$ may have a positive lifetime $\zeta_w\in(0,\infty]$  even when $w(0)=0$. 

We further impose the absence of positive jumps, i.e. $\Delta w(t)\coloneqq w(t)-w(t-)\leq 0$ for all $t\in(0,\zeta_w)$, and that $w$ never reaches the absorbing state $0$  by a jump (i.e.  $w(\zeta_w-)=0$ whenever $\zeta_w<\infty$).   We fix some deterministic procedure for enumerating the  absolute values of the jump
sizes. When $\zeta_w<\infty$, we usually decide to enumerate those in the non-decreasing order, but other procedures could also be used. When $w$ has only finitely many jumps, 
 we agree for definitiveness to complete this finite sequence by an infinite sequence of $0$'s. 
 
Functions $w$ as above will be often referred to as trajectories, and the space of trajectories is endowed with the Skorokhod $J_1$ topology. We denote by $X=(X(t))_{t\geq 0}$ the coordinate process, that is for every $t\geq 0$, $X(t)$ stands for the map $ w \mapsto w(t)$; we define similarly $\zeta=\zeta_X: w \mapsto \zeta_w$. We also write 
$(\f_t)_{t\geq 0}$ for the canonical filtration. 

Let $\alpha\in\R$ be some fixed real number. For every $b>0$, we define the rescaled trajectory
$$w^{(b)}: t\mapsto b w(b^{\alpha}t)\qquad \text{for all }t\geq 0.$$ We further use obvious notation such as 
\begin{equation}\label{E:scalenot}
X^{(b)}(t)=bX(b^{\alpha }t): w \mapsto w^{(b)}(t)\quad \text{and} \quad X^{(b)}=(X^{(b)}(t))_{t\geq 0}.
\end{equation}

\noindent $\bullet$   {\bf Self-similar Markov processes (SSMP).} 
We assume that the space of trajectories is endowed with a family of probability measures $(P_x)_{x\geq 0}$ under which the process
$X$ is both Feller and fulfills the scaling property. In particular, $P_x(X(0)=x)=1$, and for every $b>0$, 
\begin{equation}\label{scale}
\hbox{ the law of $X^{(b)}$ under
${P}_x$ is ${P}_{bx}$.}
\end{equation}

Recall our assumption  that no trajectory can reach $0$ by a jump.
A classical result due to Lamperti shows  that under the law $P_1$, the canonical process can be expressed in the form
$X(t)=\exp(\xi(\tau_{t}))$, with $\xi=(\xi(t))_{t\geq 0}$ is a spectrally negative L\'evy process (without killing) and 
$$\tau_t= \int_0^t X^{\alpha}(s)\d s \,,\qquad t\geq 0.
$$
Observe that in this framework, there are the identities
$$\tau_t= \,\inf\left\{r\geq0: \int_{0}^{r}\exp(-\alpha \xi(s))\d s\geq
t\right\} \quad \text{and} \quad \zeta=\int_0^{\infty} \e^{-\alpha \xi(s)}\d s\qquad \text{$P_1$-a.s.}$$
When $\alpha<0$, it is known that $P_1(\zeta<\infty)=1$ if and only if $\lim_{t\to \infty} \xi(t)=-\infty$ a.s., and $P_1(\zeta<\infty)=0$ otherwise (and vice-versa for $\alpha >0$).

The law of the spectrally negative L\'evy process  $\xi$ is determined by 
its Laplace exponent $\Psi: \R_+\to \R$ via
\begin{equation}\label{EqLap}
E(\exp(q\xi(t)))=\exp(t\Psi(q))\qquad \hbox{for all $t,q\geq 0$}.
\end{equation}
In turn, the Laplace exponent 
is given by the L\'evy-Khintchin formula
\begin{equation}\label{eqpsi}
\Psi(q) \coloneqq  \frac{1}{2}\sigma^2q^2 + dq +\int_{(-\infty,0)}\left( \e^{qy}-1+q(1-\e^y)\right) \Lambda(\d y)\,,\qquad q\geq 0,
\end{equation}
where $\sigma^2\geq 0$, $d\in\R$,  and $\Lambda$ is the L\'evy measure on $(-\infty,0)$ which fulfills $\int(1\wedge y^2)\Lambda(\d y)<\infty$. 
We further assume throughout this work that the L\'evy measure has an infinite total mass $\Lambda((-\infty,0))=\infty$;
as a consequence $\xi$ is non-lattice and has infinitely many jumps a.s.

In the sequel, we  say that  $X$ is a self-similar Markov process (SSMP) under the laws $(P_x)_{x\geq 0}$, with characteristics $(\Psi, \alpha)$ and refer to Chapter 13 of \cite{Kyp} for general background on this topic.

\paragraph{Cumulant, Cram\'er's condition, and tilted SSMP.}  We next define the so-called  cumulant function
$\kappa: \R_+ \to (-\infty, \infty]$  by
\begin{eqnarray}\label{eqkappa}
\kappa(q)&\coloneqq &\Psi(q)+\int_{(-\infty,0)}(1-\e^{y})^q\Lambda(\d y) \nonumber \\
&=& \frac{1}{2}\sigma^2q^2 + dq +\int_{(-\infty,0)}\left( \e^{qy}-1+q(1-\e^y)+ (1-\e^{y})^q\right) \Lambda(\d y). 
\end{eqnarray}
The cumulant is a convex function with $\lim_{q\to\infty} \kappa(q) =\infty$, and throughout this work we make the fundamental assumption that
Cram\'er's condition holds:
\begin{eqnarray}\label{E:Cramer}
\text{ there exist $0<\omega_-<\omega_+<\infty$ such that $\kappa(\omega_{\pm})=0$ and $\kappa'(\omega_-)>-\infty$,
 }
\end{eqnarray} and also that $\kappa( \omega_+ + \varepsilon)< \infty$ for some $ \varepsilon>0$.  We also write
$$\omega_{\Delta}\coloneqq  \omega_+- \omega_-.$$

The inequality $\Psi < \kappa$ combined with convexity shows that $\Psi(q)<0$ for $q\in[\omega_-, \omega_+]$, and this forces $\lim_{t\to \infty} \xi(t)=-\infty$ a.s.
Shifting the cumulant at each of those two roots yields two important functions, namely
\begin{equation} \label{E:Phipm}
\Phi^{\pm}(q)\coloneqq \kappa(q+\omega_{\pm}), \qquad q\geq 0,
\end{equation}
which can be viewed as the Laplace exponents of two L\'evy processes with no killing, say $\eta^{\pm}$. 
We then denote by $Q^{\pm}_x$ the distribution of the SSMP with characteristics $(\Phi^{\pm}, \alpha)$ started from $x>0$. Observe that, for $\alpha <0$,  $\zeta<\infty$ and $\lim_{t\to \zeta-}X(t)=0$ almost surely under $Q_x^-$, whereas
 $\zeta=\infty$ and $\lim_{t\to \infty}X(t)=\infty$ $Q_x^+$-a.s.

\paragraph{Cell systems.} A cell system is a collection of trajectories  that describe the sizes of a family of  cells as a function of their ages, and which  is endowed with a genealogical branching structure {\em \` a la} Crump-Mode-Jagers.
Roughly speaking, every jump of a trajectory during its lifetime is interpreted as a birth-event, in the sense that a daughter cell is born at the time of a jump of her mother and the size of the daughter at birth is simply given by the (absolute) size of that jump. Let us first formalize this notion. 

As usual, the genealogy is conveniently encoded by the Ulam tree ${\mathbb U}=\bigcup_{n\geq 0}\N^n$ with $\N=\{1,2,\ldots\}$; we shall use classical notation in this setting without recalling it. An element $u\in\U$ is thus a finite (possibly empty) sequence of positive integers, which we call an individual, or a cell.  A cell system 
is then defined as a family ${\mathcal X}\coloneqq ({\mathcal X}_u, u\in\U)$ indexed by the Ulam tree, where  
each ${\mathcal X}_u=\left({\mathcal X}_u(s)\right)_{s\geq 0}$ is a trajectory.

 By a slight abuse of notation, we write $\zeta_u=\inf\{s>0: {\mathcal X}_u(s)=0\}$ for the lifetime of the trajectory ${\mathcal X}_u$, i.e. the age at which the mass of the cell $u$ is absorbed at $0$. We view $\zeta_u$ as the age at death
 of the individual $u$.  
Recall that for every individual $u\in\U$, we enumerate the negative jumps of ${\mathcal X}_u$ according to some deterministic procedure, and this yields the sizes at birth ${\mathcal X}_{u1}(0)$, ${\mathcal X}_{u2}(0), \ldots$  of the children $u1, u2, \ldots$ of $u$. Working in absolute time, we also denote by $b_u$ the birth-time of the individual
$u$, so that $b_{\varnothing}=0$ and $b_{uj}-b_u$ is the instant at which the $j$-th jump of ${\mathcal X}_u$ occurs. Similarly, we write $d_u= b_u+\zeta_u$ for the death-time of the individual $u$, hence 
 $[b_u, d_u)$ is the time-interval during which this individual is alive.

 We next equip cell systems with three families of probability measures. They share the common feature that daughter cells evolve independently one of the other and according to the dynamics of the SSMP with characteristics $(\Psi, \alpha)$. We stress that the ancestor cell $\varnothing$, often referred to as Eve, may follow different dynamics, which then fully determine the law of the system.
 In other words, these probability measures are distributions of branching processes  on a space of trajectories  which have the same branching mechanisms, but for different random initial 
states, where the initial state refers here to the trajectory of Eve. 
 
We will be primarily interested in the case when the trajectory ${\mathcal X}_{\varnothing}$ of Eve  is also given by a SSMP with characteristics $(\Psi,\alpha)$. This 
 yields a first family of probability measures $({\mathcal P}_x)_{x\geq 0}$  on the space of cell systems.  More precisely, the latter is 
defined recursively as follows. We first let the Eve cell ${\mathcal X}_{\varnothing}$ have the law $P_x$ of the SSMP with characteristics $(\Psi,\alpha)$. Next, given ${\mathcal X}_{\varnothing}$, the processes of the sizes  of cells at the first generation, ${\mathcal X}_i=({\mathcal X}_i(s), s\geq 0)$ for $ i\geq 1$, have the distribution of a sequence of independent processes with respective laws $P_{x_i}$, where $x_1, x_2, \ldots$ denotes the sequence of the positive jump  sizes of $-{\mathcal X}_{\varnothing}$, ranked according to the deterministic procedure. We continue in an obvious way for the second generation, and so on for the next generations. 

The second and the third families,  $({\mathcal Q}^-_x)_{x\geq 0}$ and $({\mathcal Q}^+_x)_{x\geq 0}$, will play a sporadic role in this work  that will be explained later on. They correspond to
 the cases when the evolution of Eve is given by SSMP now with characteristics $(\Phi^-,\alpha)$, respectively $(\Phi^+,\alpha)$. 
 We shall denote the mathematical expectation under ${\mathcal P}_x$ (respectively, under ${\mathcal Q}^{\pm}_x$)
 by ${\mathcal E}_x$ (respectively, by ${\mathcal E}^{\pm}_x$).

 It will be convenient in the sequel to introduce the scaling transformation for cells. Specifically, we write for any $b>0$
\begin{equation}\label{E:scalecell}
{\mathcal X}^{(b)}_u=\left(b{\mathcal X}_u(b^{\alpha}s)\right)_{s\geq 0} \quad \text{and} \quad {\mathcal X}^{(b)}=({\mathcal X}^{(b)}_u, u\in\U),
\end{equation}
where $\alpha\in\R$ is the same parameter that we used for rescaling trajectories in \eqref{E:scalenot}.
 Note from the scaling property \eqref{scale} that
\begin{equation}\label{E:scalecellloi}
\hbox{ the law of ${\mathcal X}^{(b)}$ under ${\mathcal P}_x $ (respectively, ${\mathcal Q}^{\pm}_x $) is ${\mathcal P}_{bx}$  (respectively, ${\mathcal Q}^{\pm}_{bx} $) .
}
\end{equation}

\paragraph{Growth-fragmentations.} The growth-fragmentation $\X=(\X(t), t\geq 0)$ associated to a cell system ${\mathcal X}$ is the process describing the sequence of the masses of cells (repeated according to their multiplicities and ranked in the non-increasing order)  that are alive as a function of the absolute  time. That is, we write
 $$\X(t)=(X_1(t), X_2(t), \ldots)= \{\!\! \{{\mathcal X}_u(t-b_u): t\in [b_u, d_u)\ \&\ u\in\U \}\!\! \}^{\downarrow},$$ 
where the notation $\{\!\! \{ \cdots \}\!\! \}^{\downarrow}$ indicates that the elements of a multiset in $(0,\infty)$ are enumerated in the non-increasing order, and completed by infinitely $0$'s in the case where this multiset is finite.

Endowing cell systems with different distributions yield different laws of growth-fragmentations. Specifically, 
we write $\P_x$  (respectively, $\Q_x^{\pm}$) for the distribution of $\X$ induced by ${\mathcal P}_x$ 
(respectively, by ${\mathcal Q}_x^{\pm}$). 
The scaling property is immediately shifted to growth-fragmentations, namely for every $b>0$
$$\hbox{ the law of $\X^{(b)}\coloneqq \left(b\X(b^{\alpha}t)\right)_{t\geq 0}$ under $\P_x $ (respectively, $\Q^{\pm}_x$) is $\P_{bx}$ (respectively, $\Q^{\pm}_{bx}$).
}$$

Information about the genealogy of cells is lost when considering only $\X$ rather than ${\mathcal X}$, 
and even though the distributions of cell systems ${\mathcal P}_{x}$ and ${\mathcal Q}^{-}_{x} $ are generally mutually singular, the law of the growth-fragmentations $\P_x$  and $\Q_x^{-}$ are actually equivalent. More precisely, $\d \Q_x^{-}= x^{-\omega_-}\A \d \P_x$, where $\A$ is the so-called intrinsic area; see the next section here and also Section 4.3 in \cite{BBCK}. 
In turn, $\Q_x^{+}$ can be thought of as a version of the growth-fragmentation conditioned on having indefinite growth.

\section{The intrinsic area of a cell system and its density} 
\label{sec:area}

 The intrinsic area $\A$ of a cell system arises as the terminal value of a remarkable martingale for an underlying branching random walk. The purpose of this section is to establish several properties of its distribution, namely Theorem  \ref{thm:intro1} and \ref{thm:intro2}. Let us briefly sketch the strategy and the tools employed. First, we recall in Section \ref{ssec:smoothing} the classical \emph{smoothing transform} satisfied by $ \A$. This is a recursive distributional equation of the form
 $$ \A \quad \mathop{=}^{(\text{in law})} \quad  \sum_{i\geq 1} \gamma_i \A_i, $$
 where $ \A_i$ are i.i.d.~copies of $\A$ also independent of the positive vector $(\gamma_i)_{i\geq 1}$ whose law will be specified in \eqref{E:lawgamma} below.  In Section \ref{ssec:regularity} we rely on the work of Liu \cite{Liu2} and check the sufficient conditions on $(\gamma_i)_{i \geq 1}$ to get existence and smoothness of the density of $\A$ (Theorem \ref{thm:intro1}). 
 
 To prove Theorem \ref{thm:intro2} in Section \ref{ssec:asymp}, we shall rely on another fixed point equation. Specifically, using a spinal decomposition, we turn the smoothing transform into a recursive distributional equation of the type 
 $$ \A^- \quad  \mathop{=}^{(\text{in law})}  \quad  \A^- U + V,$$
 where $\A^-$ is the size-biaised version of $\A$ which is independent of the pair $(U,V)$, see Lemma \ref{LB}. Such equations are known under the name of \emph{random affine equations} or \emph{perpetuity equations}, see \cite{BDM}.  Actually, the law of $\A^-$ satisfies many such equations (roughly speaking, one for each Markovian path decomposition of the Eve cell).  In the framework where there are no positive jumps for the driving self-similar Markov process, we shall stop the Eve cell at   a first passage time. This allows us to obtain a specific random affine equation as above where $U$ only takes the values $0$ or $1+ \varepsilon$ (see Lemma \ref{LB}). By letting $ \varepsilon \to 0$ this enables us to study the density of $ \A^-$. See Section \ref{sec:discussion} for a discussion concerning the potential use of other random affine equations.

\subsection{Intrinsic area and smoothing transform}
\label{ssec:smoothing}

We introduce some further notation for cell systems. For every $n\geq 0$ the point process
$${\mathcal B}(n)\coloneqq \sum_{|u|=n+1} \delta_{{\mathcal X}_u(0)}$$
records the masses at birth of cells at the $(n+1)$-th generation.
We write ${\mathcal G}(n)=\sigma \left({\mathcal X}_u: |u|\leq n\right)$ for the $\sigma$-field
generated by the  trajectories of cells
at generation at most $n$, and underline that, since the masses at birth of cells at the $(n+1)$-th generation
are given by the sizes of the jumps of trajectories of cells at the $n$-th generation, ${\mathcal B}(n)$ is ${\mathcal G}(n)$-measurable.
It will be convenient in the sequel to identify implicitly ${\mathcal B}(n)$ with the sequence of its atoms, of course repeated according to their multiplicities. Thanks to Lemma 3 in \cite{BeMGFP}, we can view ${\mathcal B}(n)$  
under ${\mathcal P}_x$ 
 as a random variable with values in the space 
$\ell^{\omega_-}_+$ of nonnegative sequences $\x=(x_1, \ldots)$ with  $\sum_{1}^{\infty} x_j^{\omega_-}<\infty$.

We recall  that under the family of laws $({\mathcal P}_x)_{x\geq 0}$,  $({\mathcal B}(n))_{ n\geq 0}$ is a  branching random walk on $(0,\infty)$ equipped with the multiplication, meaning that the image of ${\mathcal B}(n)$ by the logarithmic function forms a branching random walk on $\R$  in the usual sense. This is readily seen from the branching and self-similarity properties; see Section 3.4 in \cite{BeMGFP} for details. 

In this setting, we recall from Section 2.3 in \cite{BBCK} that Cram\'er's condition \eqref{E:Cramer}
yields two important ${\mathcal P}_x$-martingales 
$${\mathcal M}^{\pm}(n) 
\coloneqq  \sum_{|u|=n+1 }{\mathcal X}_{u}^{\omega_{\pm}}(0), \qquad n\geq 0.$$
More precisely, ${\mathcal M}^+$ has terminal value 
$$\lim_{n\to \infty} {\mathcal M}^+(n)=0\qquad {\mathcal P}_x\text{-a.s.,}$$
whereas ${\mathcal M}^-$
 is uniformly integrable under  ${\mathcal P}_x$. The latter is often referred to as the intrinsic martingale; its terminal value 
\begin{equation}\label{E:defA}
\A\coloneqq  \lim_{n\to \infty} {\mathcal M}^-(n)>0
, \qquad \text{${\mathcal P}_x$-a.s.  and in $L^1({\mathcal P}_x)$,}
\end{equation}
is called the intrinsic area. The terminology comes from the connection with certain random surfaces, see \cite{BCK, BBCK,LGR}; we also refer to \cite{GED2,GED} for  fine studies of this notion. 
We further recall from Proposition 2.2 in \cite{BBCK} that  
$${\mathcal E}_x(\A)={\mathcal E}_x({\mathcal M}^-(0))=x^{\omega_-}.$$

The distribution of the intrinsic area will play a key role in this study.
It is plain by scaling \eqref{E:scalecellloi} 
that for every $x>0$,  the law of $x^{\omega_-} \A$ under ${\mathcal P}_1$ is the same as that of  $\A$ under ${\mathcal P}_{x}$, 
and henceforth we focus on the case $x=1$ without loss of generality.   In the present setting, the smoothing transform related to $\A$  has been described by Equation (15) in \cite{BBCK} as follows.  If we 
 write $(\A_i)_{i\geq 1}$ for a sequence of i.i.d.~copies of $\A$ under ${\mathcal P}_1$, then  there is the identity in distribution
\begin{equation}\label{E:smoothA}
\text{ the law of $\A$ under ${\mathcal P}_1$ is the same as that of }  \sum_{i=1}^{\infty} \gamma_i\A_i, 
\end{equation}
where 
 \begin{equation}\label{E:lawgamma}
\text{ $(\gamma_i)_{i\geq 1}$  has the law of some enumeration of
$\{\!\! \{|\Delta{X}(t)|^{\omega_-}: 0<t<\zeta\}\!\!\}$ under $P_1$,}
\end{equation}
 and is further independent of $(\A_i)_{i\geq 1}$.
 In the rest of this section, \eqref{E:smoothA} and some related expressions will play a key role for investigating properties of the law of the intrinsic area.

   \subsection{Regularity of the law of the intrinsic area}
   \label{ssec:regularity}
   
Here we prove Theorem \ref{thm:intro1}, which establishes in particular the existence of a smooth density for $\A$.

\begin{proof}[Proof of Theorem \ref{thm:intro1}] The proof relies  on Liu \cite{Liu2}. We note that Liu considers smoothing transforms in which the series has only finitely many terms a.s.; however, as far as the results that are needed here are concerned,  his arguments work just as well for infinite series. 
Using Theorem 2.1 in  \cite{Liu2}, 
we shall prove that the characteristic function $\phi(\theta)={\mathcal E}_1(\exp({\mathrm i}\theta \A))$, $\theta\in\R$,  fulfills 
$\phi(\theta)=O(|\theta|^{-b})$ as $|\theta| \to \infty$ for every $b>0$. This ensures the existence of a density in  ${\mathcal C}^{\infty}_0(\R)$
by standard Fourier analysis; since the area is a nonnegative random variable, this density can be viewed as a function in ${\mathcal C}^{\infty}_0(\R^*_+)$.

Our standing assumptions guarantee
that the SSMP $X$ makes infinitely many jumps $P_1$-a.s., so $\gamma_j>0$ a.s. for all $j\geq 1$ and 
 {\em a fortiori}  extinction never occurs for the branching random walk $({\mathcal B}(n))_{n\geq 0}$. Condition (2.1) in \cite{Liu2} is thus fulfilled.
Recall also that ${\mathcal E}_1(\A)={\mathcal E}_1(\sum_{j\geq 1} \gamma_j)=1$, which is Condition (2.2) in \cite{Liu2}. 
The core of the proof now amounts to checking the first part of Condition (2.3) in \cite{Liu2}, as the second part is trivially fulfilled in our setting. 

Specifically, we have to prove that 
\begin{equation} \label{E:mom-b}
{\mathcal E}_1(\gamma_1^{-b})<\infty \qquad \text{for any }b>0.
\end{equation}
In this direction, we first note that, since time-substitution does not alter the sizes of jumps,  the Lamperti transformation for SSMP implies that we may choose 
$$\gamma_1=\sup\{|\Delta \e^{\xi}(t)|^{\omega_-}: t\geq 0\}.$$
Denote the (left) tail of the L\'evy measure by $\bar \Lambda(y)=\Lambda((-\infty,y))$ for $y<0$.
The first instant $T(y)=\inf\{t\geq 0: \Delta \xi(t) <y\}$ when $\xi$ makes a jump with (relative) size less than $y$, has an exponential distribution with parameter $\bar \Lambda(y)$, and is further independent of the process $\xi_y$ that obtained from $\xi$ by suppressing all its jumps $\Delta \xi(t)$ with $\Delta \xi(t)<y$. 

On the one hand, since $\xi(T(y)-)= \xi_y(T(y))$, 
there is the lower bound
$$\gamma_1 \geq (1-\e^y)^{\omega_-}\exp(\omega_- \xi_y(T(y))).$$
On the other hand, $\xi_y$ is a spectrally negative  L\'evy process with Laplace exponent
$$\Psi_y(q) \coloneqq  \frac{1}{2}\sigma^2q^2 + dq +\int_{[y,0)}\left( \e^{qx}-1+q(1-\e^x)\right) \Lambda(\d x)
+q \int_{(-\infty, y)}(1-\e^x) \Lambda(\d x)\,,\qquad q\in \R.$$
We stress that this quantity is finite for all $q\in\R$, and we have  $E(\exp(q\xi_y(t)))=\exp(t\Psi_y(q))$ for all $t\geq 0$ and $q\in\R$.
We take $q=-b\omega_-<0$ and observe that the map $y\mapsto \Psi_y(q)$ is non-increasing on $(-\infty,0)$. 
Since $\lim_{y\to 0+} \bar \Lambda(y)=\infty$, we may choose $y<0$ close enough to $0$ so that 
$\bar \Lambda(y)>\Psi_y(q)$. This yields 
$$E\left(\exp(-b\omega_- \xi_y(T(y)))\right)=1/(\bar \Lambda(y)-\Psi_y(q))<\infty,$$
and  we conclude that  ${\mathcal E}_1(\gamma_1^{-b})<\infty$.
\end{proof}

\subsection{Asymptotic behavior of the density}
\label{ssec:asymp}

The goal of this section is to establish Theorem \ref{thm:intro2}. As explained in the beginning of this section, we shall work under ${\mathcal Q}^-_1$, the distribution of the cell system when the trajectory of the Eve cell has the law 
$Q^-_1$ (recall that $Q^-_x$ denotes the distribution of the SSMP with characteristics $(\Phi^-,\alpha)$ which is associated to  the spectrally negative L\'evy process $\eta^-$ by the Lamperti transformation) whereas any cell at generation at least $1$ and started with mass $x>0$  follows the law $P_x$. The main advantage of working under ${\mathcal Q}^-_1$ is that  when one splits the contribution to the intrinsic area of the Eve cell before and after  a first passage time, one gets a tractable random affine equation thanks to path decompositions for L\'evy processes with no positive jumps (see Lemma \ref{LB}). The proof  of Theorem \ref{thm:intro2} then consists in analyzing  infinitesimal first passage times.

\bigskip

More precisely, note that the almost sure convergence in \eqref{E:defA} also holds under  ${\mathcal Q}^-_1$, because the (branching) transitions probabilities of $({\mathcal B}(n))_{n\geq 0}$ are the same under  $({\mathcal P}_x)_{x> 0}$ as under $({\mathcal Q}^-_x)_{x> 0}$; only the initial distribution of ${\mathcal B}(0)$ is different.
Beware however that uniformly integrability may fail under  ${\mathcal Q}^-_1$, simply because 
the initial variable ${\mathcal M}^-(0)$ may have an infinite expectation and then also ${\mathcal E}^-_1(\A)=\infty$ (we see from  the forthcoming Lemma \ref{LA} that ${\mathcal E}^-_1(\A)={\mathcal E}_1(\A^2)$, 
and thus from \eqref{E:tailA} that 
${\mathcal E}^-_1(\A)<\infty$ if and only if $\omega_+> 2\omega_-$).  In this setting, the counterpart of \eqref{E:smoothA} reads
\begin{equation}\label{E:smoothA-}
\text{ the law of $\A$ under ${\mathcal Q}^-_1$ is the same as that of } \A^-\coloneqq \sum_{i=1}^{\infty} \gamma^-_i\A_i, 
\end{equation}
where $(\gamma^-_i)_{i\geq 1}$  is  independent of $(\A_i)_{i\geq 1}$ and  has the law of some enumeration of
$\{\!\! \{|\Delta{X}(t)|^{\omega_-}: 0<t<\zeta\}\!\!\}$ under $Q^-_1$.  We stress that in \eqref{E:smoothA-}, the 
$\A_i$ are i.i.d.~versions of the intrinsic area under ${\mathcal P}_1$ (not under ${\mathcal Q}^-_1$ !) and further note from Lamperti's transformation that
\begin{equation}\label{E:lawgamma-}
\text{ $(\gamma^-_i)_{i\geq 1}$  has the law of some enumeration of
$\{\!\! \{|\Delta \e^{\eta^{-}}(t)|^{\omega_-}: t > 0\}\!\!\}$.}
\end{equation}

We first point at a simple relation between the distribution of the intrinsic area under ${\mathcal P}_1$ and under ${\mathcal Q}^-_1$.
Recall the notation $\omega_{\Delta}\coloneqq  \omega_+- \omega_-$ and that $c$ is the constant appearing in \eqref{E:tailA}.

\begin{lemma} \label{LA} The distribution of $\A^-$, that is that of $\A$  under ${\mathcal Q}^-_1$
is the size-biased of that under ${\mathcal P}_1$. Specifically, one has
$${\mathcal Q}^-_1(\A\in \d r) =a^-(r) \d r  \quad \text{ with }a^-(r)\coloneqq r a(r), r\in\R,$$
and as a consequence, 
$${\mathcal Q}^-_1(\A> r) \sim c \frac{\omega_+}{\omega_{\Delta}}r^{-\omega_{\Delta}/\omega_-}\qquad \text{as }r\to \infty.$$ 
\end{lemma}

\begin{proof} The first claim  is an immediate consequence of the  so-called 
spinal decomposition for cell systems 
under the tilted probability measure $ \A{\mathcal P}_1$;  see
Section 4.3 in \cite{BBCK} and more precisely Theorem 4.7 there.
The second assertion then follows from Theorem \ref{thm:intro1}, and the third 
 by combination  with \eqref{E:tailA}. 
\end{proof}

In short, Lemma \ref{LA} enables
us to rephrase Theorem \ref{thm:intro2} in terms of the variable
$\A^-$, and we shall analyze the distribution of the latter by 
combining   \eqref{E:smoothA-}  and \eqref{E:lawgamma-} with some well-known properties of  the spectrally negative L\'evy processes $\eta^{\pm}$.  We refer to Section 8.1 in \cite{Kyp} for background from which the   assertions below can be inferred.

We first point from \eqref{E:Cramer} and \eqref{E:Phipm}  at the identities 
\begin{equation} \label{E:relphi}
\Phi^-(\omega_{\Delta})=0\quad\text{and} \quad \Phi^+(q)=\Phi^-(\omega_{\Delta} +q) \quad\text{for all }q\geq 0.
\end{equation}
Clearly,   $\eta^{\pm}$ drifts to $\pm\infty$  in the sense that $\lim_{t\to \infty} \eta^{\pm}(t)={\pm}\infty$  a.s., and roughly speaking, $\eta^+$ should be thought of as a version of $\eta^-$
conditioned to drift to $+\infty$. Rigorously, introduce for $x>0$ the first passage time above the level $x$,
$$t^{\pm}(x)\coloneqq \inf\{t>0: \eta^{\pm}(t)>x\}.$$
Then the process $(t^-(x))_{x\geq 0}$ is a subordinator killed at rate $\omega_{\Delta}$, whereas  $(t^+(x))_{x\geq 0}$
is a subordinator with finite exponential moment of some positive order. In particular, we have for every $x>0$
\begin{equation}\label{E:tempspass}
P(t^-(x)<\infty) = \e^{-x \omega_{\Delta}} \quad\text{and} \quad E(\exp(\theta t^+(x))) =\exp(x \varrho(\theta))<\infty \quad\text{for some } \theta>0,
\end{equation}
where $\varrho(\theta)>0$ solves  $\Phi^+(-\varrho(\theta))=\Phi^-(\omega_{\Delta}-\varrho(\theta))= -\theta$.
We also stress that, due to the absence of positive jumps,  whenever the first passage above $x$ takes place, it must occur continuously. Finally, $(\eta^+(t))_{0\leq t < t^+(x)}$ has the same distribution as 
$(\eta^-(t))_{0\leq t < t^-(x)}$ conditioned on $t^-(x)<\infty$. 

 Applying the strong Markov property for $\eta^-$ at time $t^-(x)$ conditionally on the event $t^-(x)<\infty$  to the decomposition  $$\{\!\! \{|\Delta \e^{\eta^{-}}(t)|^{\omega_-}: t > 0\}\!\!\} = \{\!\! \{|\Delta \e^{\eta^{-}}(t)|^{\omega_-}: 0<t< t^-(x)\}\!\!\} \sqcup \{\!\! \{|\Delta \e^{\eta^{-}}(t)|^{\omega_-}: t > t^-(x)\}\!\!\},$$
we immediately deduce from the facts recalled above the following  random  affine equation, (which we write in a conditional form for simplicity).

\begin{lemma}\label{LB} Fix $x>0$. Keeping the same notation as in \eqref{E:smoothA-} and \eqref{E:lawgamma-}, we have
$$ \text{ the conditional distribution of $\A^-$ given $t^-(x)<\infty$
is the same as that of $\A^{+}(x)+ \e^{x\omega_-} \A^-$, }$$
where  $\A^{+}(x)$ and $\A^-$ are independent variables,  
$$\A^{+}(x) \coloneqq \sum_{i=1}^{\infty} \gamma^{+}_i(x)\A_i,$$
with $(\A_j)_{j\geq 1}$ a sequence of i.i.d.~copies of $\A$ under ${\mathcal P}_1$ and $(\gamma^+_i(x))_{i\geq 1}$  an independent sequence which has the law of some enumeration of
$\{\!\! \{|\Delta \e^{\eta^{+}}(t)|^{\omega_-}: 0< t < t^+(x)\}\!\!\}$.  

\end{lemma}

Taking into account \eqref{E:tempspass}, if follows that, for every $x,r>0$:
\begin{equation}
\label{eq:tailA}
 {\mathcal P}(\A^->r) = \e^{-x \omega_{\Delta}}  {\mathcal P}(\A^+(x)+\e^{x\omega_-}\A^->r) + \left(1-\e^{-x \omega_{\Delta}} \right) {\mathcal P}(\A^->r\mid t^-(x)=\infty).
\end{equation}
This identity will be at the heart of the proof of Theorem \ref{thm:intro2}, which will consist in first taking $x \to 0+$ and then $r \to + \infty$.

We will need the following technical lemma, whose proof is postponed to the end of this section.

\begin{lemma} \label{LC} For every $1\leq p < \omega_+/\omega_-$, we have:
\begin{enumerate}
\item[(i)]  $\displaystyle{\mathcal E}\left( (\A^+(x))^p\right) = O(x)$ as $x\to 0+$;
\item[(ii)]
$\displaystyle \limsup_{x\to 0+} \frac{1}{x} {\mathcal P}(\A^+(x)>R- \e^{x \omega_-}\A^-\geq 0)
 = o(R^{-p}) \qquad \text{as }R\to \infty.$
\end{enumerate}
\end{lemma}

We continue our analysis by considering the asymptotic behavior of the conditional law of $\A^-$ given 
$t^-(x)=\infty$   as $x\to 0+$. This relies on some features on path decompositions of L\'evy processes without positive jumps at their overall supremum. We introduce 
$$\varsigma^-=\sup\{\eta^-(t): t\geq 0\} \quad \text{and}\quad  {v}=\inf\left\{t\geq 0: \sup_{0\leq s \leq t}\eta^-(s)=\varsigma^-\right\}$$ 
 for the overall supremum of $\eta^-$ and  the (first) instant when it is reached. 
 Writing 
 $$\vec{\eta}=\left(\vec{\eta}(s)\coloneqq \eta^{-}({v}+s)-\varsigma^-\right)_{ s\geq 0}$$
 for the post-supremum process. We re-express \eqref{E:lawgamma-} using the decomposition 
$$\{\!\! \{|\Delta \e^{\eta^{-}}(t)|^{\omega_-}: t > 0\}\!\!\} = \{\!\! \{|\Delta \e^{\eta^{-}}(t)|^{\omega_-}: 0<t< v\}\!\!\}
\sqcup \{\!\! \{\e^{\omega_-\varsigma^-}|\Delta \e^{\vec{\eta}}(t)|^{\omega_-}: t \geq  0\}\!\!\}$$
(we stress  that 
$\vec{\eta}(0)=\eta^-({v})-\varsigma^-=0$ a.s. if and only if the L\'evy process $\eta^-$ has unbounded variation; 
otherwise $\vec{\eta}(0)<0$ a.s. and we  view $t=0$ as a jump time of the trajectory by agreeing that 
$\vec{\eta}(0-)=0$). 
It is known from \cite{Be91} that
the pre-supremum process 
 $(\eta^-(t))_{ 0\leq t < {v}}$ and the post-supremum process $\vec{\eta}$ are  independent. More precisely, the pre-supremum process has the same law as $\eta^+$ killed when it reaches an independent exponential level with parameter $\omega_{\Delta}$.
In turn, the post-supremum process can be viewed as the version of $\eta^-$ conditioned to stay negative, 
i.e. the limit in distribution of $\eta^-$ conditioned on $\varsigma^-<\varepsilon$ as $\varepsilon \to 0+$.

\begin{lemma} \label{LD}  Introduce the variable 
$$\vec{\A} \coloneqq \sum_{i=1}^{\infty} \vec{\gamma}_i\A_i,$$
where $(\A_j)_{j\geq 1}$ is a sequence of i.i.d.~copies of $\A$ under ${\mathcal P}_1$, 
and $(\vec{\gamma}_i)_{i\geq 1}$ an independent sequence which has the law of some enumeration of
$\{\!\! \{|\Delta \e^{\vec{\eta}}(t)|^{\omega_-}: t> 0\}\!\!\}$. We then have:

\begin{enumerate}
\item[(i)] $\displaystyle \lim_{x\to  0+} {\mathcal P}(\A^->r\mid \varsigma^-\leq x) = {\mathcal P}(\vec{\A}>r)$ for every $r>0$,
\item[(ii)] $ \displaystyle r \omega_- a^-(r) =  \omega_{\Delta} \left( {\mathcal P}(\A^->r)-  {\mathcal P}(\vec{\A}>r)\right)
- \limsup_{x\to 0+} \frac{1}{x} {\mathcal P}(\A^+(x)>r- \e^{x \omega_-}\A^-\geq 0),$

\item[(iii)]
$ \displaystyle \lim_{r\to \infty} r^{\omega_{\Delta}/\omega_-} {\mathcal P}(\vec{\A}>r)=0$.
\end{enumerate} 
\end{lemma}

\begin{proof} (i) The path decomposition of $\eta^-$ at its overall supremum that has been presented above entails that 
the conditional distribution of $\A^-$ given $\varsigma^-=x$ is the same as $\A^+(x) + \e^x \vec{\A}$. 
Since $\lim_{x\to 0+}\A^+(x)=0$ in probability (this is seen e.g. from Lemma \ref{LC} (i)), the first assertion follows. 

(ii)  Rewrite \eqref{eq:tailA} as $$ {\mathcal P}(\A^->r) = \e^{-x \omega_{\Delta}}  {\mathcal P}(\A^+(x)+\e^{x\omega_-}\A^->r) + \left(1-\e^{-x \omega_{\Delta}} \right) {\mathcal P}(\A^->r\mid \varsigma^-\leq x).$$
Hence
$$ {\mathcal P}(\A^-\in(\e^{-x\omega_{-}}r, r]) =  \left(\e^{x \omega_{\Delta}}-1 \right) \left(  {\mathcal P}(\A^->r) - {\mathcal P}(\A^->r\mid \varsigma^-\leq x)\right) -{\mathcal P}(\A^+(x)>r- \e^{x \omega_-}\A^-\geq 0) . $$
Dividing by $x$ and then letting $x\to 0+$, we get (i) from the definition of $a^-$ in Lemma \ref{LA}.

(iii) We  claim that 
\begin{equation}\label{E:intarrow}
\lim_{R\to \infty} R^{\omega_{\Delta}/\omega_-} \int_R^{\infty} {\mathcal P}(\vec{\A}>r) \frac{\d r}{r} =0,
\end{equation}
from which the statement follows easily. Indeed, it suffices then to observe that, since the function $r\mapsto  {\mathcal P}(\vec{\A}>r)/r$ decreases, 
$$R^{\omega_{\Delta}/\omega_-} {\mathcal P}(\vec{\A}>R) \leq  2R^{\omega_{\Delta}/\omega_-} \int_{R/2}^{R} {\mathcal P}(\vec{\A}>r) \frac{\d r}{r} ,$$
and that the right-hand side goes to $0$ as $R\to \infty$ thanks to \eqref{E:intarrow}.

We need to check \eqref{E:intarrow}. By  (ii), for every $r>0$ we have 
$$\frac{ {\mathcal P}(\vec{\A}>r)}{r} \leq \frac{ {\mathcal P}(\A^->r)}{r} - \frac{\omega_-} {\omega_{\Delta}}a^-(r),$$
and it suffices to recall from Lemma \ref{LA} that as $R\to \infty$, one has 
$$ \int_R^{\infty} {\mathcal P}(\A^{-}>r) \frac{\d r}{r} \sim c \frac{\omega_+ \omega_-}{\omega_{\Delta}^2} R^{-\omega_{\Delta}/\omega_-}\ \text{and} \ \int_R^{\infty}a^-(r)\d r= {\mathcal P}(\A^{-}>R) \sim c \frac{\omega_+}{\omega_{\Delta}} R^{-\omega_{\Delta}/\omega_-}.
$$
This completes the proof.
\end{proof}

We are now ready to establish Theorem \ref{thm:intro2}.

\begin{proof}[Proof of Theorem \ref{thm:intro2}] We start with (i). By Lemma \ref{LD} (ii), (iii) and Lemma \ref{LC} (ii), we have $$\lim_{r\to \infty} r^{\omega_+/\omega_-} a^-(r)=   c\frac{\omega_+}{\omega_-},$$
where $c$ is the constant appearing in \eqref{E:tailA}.  The desired result follows from  Lemma \ref{LA}.

It remains to check that $a(r)>0$ for every $r>0$. In this direction, we first observe from Lemma \ref{LA} and the first part of the proof that we already know that $a(r)>0$ when $r$ is sufficiently large, say $r\geq R$. Fix any $r\in(0,R)$, recall the smoothing transform  \eqref{E:smoothA} and let $(g_i)_{i\geq 1}$ be any sequence 
of strictly positive real numbers with $\sum_{i\geq 1} g_i <\infty$ . We choose
$j\in\N$ sufficiently large such that $0<g_j<r/(2R)$.  The variable $g_j\A_j$ has density $g_j^{-1} a(\cdot/g_j)$ on $(0,\infty)$ and is independent of $\sum_{i\neq j} g_i \A_i $.
Thus for every $\varepsilon \in(0,1)$, we have
$${\mathcal P}\left(\sum_{i} g_i \A_i \in[r,r+\varepsilon)\right) 
\geq \frac{ \varepsilon }{g_j} \inf\{a(s): r/2\leq sg_j \leq r+1\} {\mathcal P}\left(\sum_{i\neq j} g_i \A_i < r/2\right)
\,.$$
On the one hand, $r/(2g_j)>R$ and therefore $\inf\{a(s): r/2\leq s g_j \leq r+1\}>0$. On the other hand, Theorem 2 of Biggins and Grey \cite{BigGrey} ensures that     ${\mathcal P}(\A_i<b)>0$ for every $b>0$ and $ i \geq 1$. It easily follows  that 
$${\mathcal P}\left(\sum_{i\neq j} g_i \A_i < r/2\right) >0\,.$$
We conclude that 
$$\liminf_{\varepsilon \to 0+} \frac{1}{ \varepsilon }{\mathcal P}\left(\sum_{i} g_i \A_i \in[r,r+\varepsilon)\right)  >0.$$
By conditioning the smoothing transform \eqref{E:smoothA} on the sequence $(\gamma^{\omega_-}_i)_{i\geq 1}=(g_i)_{i\geq 1}$ and applying Fatou's lemma,
we now see that indeed $a^-(r)>0$. 
\end{proof}

We conclude this section with the proof of Lemma \ref{LC}.

\begin{proof}[Proof of Lemma \ref{LC}] We first note, that, since  $\sup_{t>0} t^p /(\e^{\theta t}-1)<\infty$ for any $\theta >0$ and any $p\geq 1$,  
 \eqref{E:tempspass}  yields
 \begin{equation}\label{E:tempspassmom} 
 E((t^+(x))^p) = O(x) \qquad \text{as }x\to 0+.
 \end{equation}

By the L\'evy-It\=o decomposition, the point process describing the jumps of the L\'evy process $\eta^+$ is Poisson with intensity $\d t \Pi^+(\d y)$, where $\Pi^+$ denotes the L\'evy measure of $\eta^+$. We mark further each jump, say, $(t,\Delta \eta^+(t))$, with an independent variable $\A(t)$ having the law of $\A$ under ${\mathcal P}_1$, and  obtain a Poisson point process with intensity $\d t \Pi^+(\d y) a(r) \d r$. In this setting, we consider the process
$$N(t)\coloneqq \sum_{0< s \leq  t}|\Delta \e^{\eta^{+}}(s)|^{\omega_-}\A(s)= \sum_{0< s \leq t}\exp({\omega_-\eta^{+}}(s-))(1-\e^{\Delta \eta^{+}(s)})^{\omega_-}\A(s), \qquad t\geq 0.$$
Note that only instants $s$ at which $\eta^+$ jumps contribute to the sum, 
and  that $\A^+(x)$ has the same law as $N(t^+(x))$. 

Next observe from Lemma 2.1(i) in \cite{BBCK} and the fact that $\kappa(p\omega_-)<\infty$ that
$$c_p\coloneqq \int_{(-\infty,0)} (1-\e^y)^{p\omega_-}\Pi^+(\d y) <\infty.$$
Recall also from \eqref{E:tailA} that  ${\mathcal E}_1(\A^p)<\infty$ since $p<\omega_+/\omega_-$, and 
that for $p=1$, ${\mathcal E}_1(\A)=1$. 
This allows us to also consider  the compensated sum
$$N^{(c)}(t)\coloneqq N(t) - c_1 \int_0^t\exp({\omega_-\eta^{+}}(s))\d s, \qquad t\geq 0,$$
which is a purely discontinuous martingale. 
Plainly, 
$$\int_0^{ t^+(x)}\exp({\omega_-\eta^{+}}(s)) \d s\leq  t^+(x) \e^{x\omega_-},$$
and  \eqref{E:tempspassmom} entails
$$E\left(\left(\int_0^{ t^+(x)}\exp({\omega_-\eta^{+}}(s)) \d s\right)^p\right) = O(x)\qquad \text{as }x\to 0+.$$
So to complete the proof, we just need to establish that 
$$E((N^{(c)}(t^+(x))^p)=O(x)\qquad \text{as }x\to 0+.$$ 

In this direction, we  shall use Burkholder-Davis-Gundy inequalities for the martingale $N^{(c)}$
and need therefore to show that
$$E\left( [N^{(c)}]^{p/2}(t^+(x))\right) =O(x)\qquad \text{as }x\to 0+,$$ 
where $[N^{(c)}]$ stands for the quadratic variation of the purely discontinuous martingale $N^{(c)}$. 
We first suppose $p\leq 2$, and use the bound
$$[N^{(c)}]^{p/2} (t) = \left(\sum_{0< s \leq t} |\Delta N^{(c)}(s)|^2\right)^{p/2} \leq 
\sum_{0< s \leq t} |\Delta N^{(c)}(s)|^p
= \sum_{0< s \leq t}\exp({p\omega_-\eta^{+}}(s-))(1-\e^{\Delta \eta^{+}(s)})^{p\omega_-}\A^p(s).$$
A calculation by compensation similar to that at the beginning of the proof shows that 
the expectation of this quantity evaluated for $t=t^+(x)$ is bounded from above by
$$\e^{xp\omega_-}c_p{\mathcal E}_1(\A^p)E((t^+(x))^p),$$
and again thanks to \eqref{E:tempspassmom}, this quantity is $O(x)$.

The case $p\in(2,4]$ is mostly similar. We first compensate $ [N^{(c)}]$ to get a martingale, and
proceed as above. Iteratively one deals with any $p\in(1, \omega_+/\omega_-)$. Details are left to the reader, we also refer to Lemma 2.3 in \cite{BBCK} for a similar argument. 

For the second assertion,  recall from Lemma \ref{LA} that $\A^-$ has  density $a^-$, and from Lemma \ref{LB} that $\A^+(x)$ and $\A^-$ are independent. For every $R>0$, we have
$$
{\mathcal P}(\A^+(x)>2R- \e^{x \omega_-}\A^-\geq 0) = \e^{x\omega_-}\int_0^{2R}  a^-(r\e^{x\omega_-}) {\mathcal P}(\A^+(x)>2R-r)\d r.\\
$$
Splitting the integral at $R$ and then using the change of variables $s=2R-r$
shows that the right-hand side is bounded from above by 
%\begin{align*}
%& \ {\mathcal P}(\A^+(x)>R) + \e^{x\omega_-}  \int_0^{R} a^-((2R-s)\e^{x\omega_-}) {\mathcal P}(\A^+(x)>s)\d s \\
%&\leq  {\mathcal P}(\A^+(x)>R) + \sup_{r\geq R}a^-(r)\times  {\mathcal E}(\A^+(x)). 
%\end{align*}
$$ {\mathcal P}(\A^+(x)>R) + \e^{x\omega_-}  \int_0^{R} a^-((2R-s)\e^{x\omega_-}) {\mathcal P}(\A^+(x)>s)\d s 
\quad	\leq \quad  {\mathcal P}(\A^+(x)>R) + \sup_{r\geq R}a^-(r)\times  {\mathcal E}(\A^+(x)). 
$$
Then take any  $p'\in(p, \omega_+/\omega_-)$ and note from  Markov's inequality and (i) that 
$${\mathcal P}(\A^+(x)>R)\leq R^{-p'}  {\mathcal E}((\A^+(x))^{p'})
\quad \text{and} \quad   {\mathcal E}((\A^+(x))^{p'}) =   O(x) \qquad \text{as  $x\to 0+$ and $R\to \infty$} .$$
Therefore, by (i), to finish the proof, it remains to check that
\begin{equation}
\label{eq:limsup}
\limsup_{r\to \infty} r^{\omega_+/\omega_-} a^-(r) \leq   c{\omega_+}/{\omega_-}.
\end{equation}

To this end, we use \eqref{eq:tailA} to get that, for every $\varepsilon>0$, $ {\mathcal P}(\A^->r) \geq \e^{-\varepsilon \omega_{\Delta}}  {\mathcal P}(\e^{\varepsilon\omega_{-}}\A^->r) $,
and deduce the upper-bound 
$$ {\mathcal P}(\A^-\in(\e^{-\varepsilon\omega_{-}}r, r]) \leq  \left(\e^{\varepsilon \omega_{\Delta}}-1 \right)  {\mathcal P}(\A^->r). $$
Dividing the expression above by $\varepsilon$ and then letting $\varepsilon\to 0+$ gives 
that for all $r>0$, $ r a^-(r) {\omega_-} \leq {\omega_{\Delta}} {\mathcal P}(\A^->r)$.  The inequality \eqref{eq:limsup} then follows from Lemma \ref{LA}. This completes the proof.
\end{proof}

\subsection{An alternative approach to Theorem \ref{thm:intro2}} \label{sec:discussion}
The above proof of Theorem \ref{thm:intro2} crucially relies on the absence of positive jumps for the driving self-similar Markov processes. In view of a generalization of Theorem \ref{thm:intro2}, we pave here a possible route in the  presence of positive jumps, and more generally for variables satisfying a ``nice'' random affine equation. We do not carry the details and stay at the level of a discussion.  
\medskip

Let us consider a random affine equation of the form
 \begin{eqnarray} \A^- &\overset{( \mathrm{d})}{=}& \A^- U + V,  \end{eqnarray}
 where $\A^-$ is independent of the pair $(U,V)$. We assume that $\A^-$ has a density (see \cite{leckey2019densities} for  sufficient conditions), denoted by $a^-$. This time, instead of focusing on the case where $U$ only takes two values as in Section \ref{ssec:asymp}, we shall rather focus on the case where \begin{center}$(U,V)$ has a continuous density over $ \mathbb{R}_+^2$. \end{center}  In the case of growth-fragmentations and under mild assumptions on the driving SSMP, this should be achievable by stopping the Eve particle at a fixed time, or when it overshoots a random level. As in Section \ref{ssec:asymp}, let us suppose that $ \mathbb{E}[U^\rho] =1$ (and also  $ \mathbb{E}[U^{\rho + \varepsilon}] < \infty$ and $ \mathbb{E}[V^{\rho+ \varepsilon}] < \infty$) with $\rho = \frac{\omega_\Delta}{\omega_-}$ so that we are in position to apply the Kesten-Grincevi\v cius-Goldie theorem (see \cite[Theorem 2.4.4]{BDM}), and recover Lemma \ref{LA}:
 $$ \int_R^\infty \mathrm{d}r\,  a^-(r) = \mathcal{P}( \A^- > R) \sim  \frac{c}{\rho} R^{-\rho}.$$ Now, if $f(u,v)$ denotes the joint density of $(U,V)$, the random affine equation shows that for any positive measurable function $\psi$ we have 
 $$ \int_{0}^{\infty} \mathrm{d}r \, a^-(r) \psi(r) = \int_{0}^{\infty} \mathrm{d}s \int_{0}^{\infty} \mathrm{d}u \int_{0}^{\infty} \mathrm{d}v  \ \psi(su + v) f(u,v) a^-(s),$$
 so that after performing the change of variable $r=su+v$ and eliminating $u$ in the right-hand side we deduce the following integral equation for the density of $a^-(r)$:
 \begin{eqnarray} a^-(r) &=& \int_{0}^{\infty} \mathrm{d}s  \frac{a^-(s)}{s}\int_{0}^{\infty} \mathrm{d}v f\left( \frac{r-v}{s}, v\right).  \end{eqnarray}
Our goal is now to use the continuity of $f$ to ``average" the possible roughness of the density $a^-$ in the right-hand side.

Let us show how to deduce an asymptotic lower bound on $a^-$ using this approach. First, for any $ \varepsilon \in (0,1)$ by the above display the quantity $a^{-}(r)$ is bounded from below by $$ \int_{ \varepsilon r}^{ \varepsilon^{-1} r} \mathrm{d}s  \frac{a^-(s)}{s}\int_{ \varepsilon}^{ \varepsilon^{-1}} \mathrm{d}v \, f\left( \frac{r-v}{s}, v\right) \underset{\fbox{$s=zr$}}{=}  r^{-\rho-1}\int_{ \varepsilon}^{ \varepsilon^{-1}} \mathrm{d}z \, \underbrace{r^{\rho+1}a^-(zr) z^{-1}}_{=: g_r(z)} \underbrace{\int_{ \varepsilon}^{ \varepsilon^{-1}}  \mathrm{d}v \, f\left( \frac{r-v}{z r}, v\right)}_{=: F_r(z)}.$$ 
Using the continuity of $f$, as $r \rightarrow \infty$, we deduce that $F_r(z)$ converges uniformly on $( \varepsilon, \varepsilon^{-1})$ towards the continuous function $F(z) = \int_{ \varepsilon}^{ \varepsilon^{-1}}  \mathrm{d}v f(z^{-1},v)$.  On the other hand, the asymptotics $\mathcal{P}( \A^- > R) \sim  \frac{c}{\rho} R^{-\rho}$ shows that $g_r(z)$ converges weakly towards $g(z)= c  z^{-\rho-2}$ in the sense that for every continuous function $\phi$ we have 
 $$ \int_ \varepsilon^{ \varepsilon^{-1}} \mathrm{d}z\,  g_r(z) \phi(z)   \quad \mathop{\longrightarrow}_{r \rightarrow \infty} \quad  \int_ \varepsilon^{ \varepsilon^{-1}} \mathrm{d}z\,  g(z) \phi(z).$$ 
Gathering those two convergences  in the penultimate display we deduce that
$$ \liminf_{r \to \infty} a^-(r) r^{\rho+1} \geq \lim_{r \to \infty} \int_{ \varepsilon}^{ \varepsilon^{-1}} \mathrm{d}z\ g_r(z) F_r(z) = \int_{ \varepsilon}^{ \varepsilon^{-1}} \mathrm{d}z\, g(z) F(z) = c \int_{ \varepsilon}^{ \varepsilon^{-1}} \mathrm{d}z \, z^{-\rho-2} \int_{ \varepsilon}^{ \varepsilon^{-1}} \mathrm{d}v f(z^{-1}, v).$$
Letting $ \varepsilon \to 0$ in the last integral we recognize $ \mathbb{E}[U^\rho]=1$ after the change of variable $u= 1/z$. This indeed shows that $a^-(r) \geq {c} r^{-\rho-1}$ asymptotically as desired. Getting an asymptotic upper bound on $a^-$ needs more assumptions on $f$ and we shall leave this for further research.

\subsection{Series of independent intrinsic areas}

In this section, we will turn our attention to infinite weighted sums of  i.i.d.~copies of $\A$. This will be useful in order to consider other initial conditions, as will be needed in Section \ref{sec:conditioning}.

Recall that $(\A_i)_{i\geq 1}$ denotes a sequence of i.i.d.~copies of $\A$ under
${\mathcal P}_{1}$ and $\ell^{\omega_-}_+$ the space of sequences  $\x=(x_1, \ldots)$ with nonnegative terms such that 
$\sum_i x_i^{\omega_-}<\infty$. Since ${\mathcal E}_1(\A)=1$, the series
$$\A_{\x} \coloneqq\sum_{i=1}^{\infty} x_i^{\omega_-}\A_i$$
converges a.s. and in $L^1({\mathcal P})$ for every $\x\in \ell^{\omega_-}_+$, and the purpose of this section is to investigate the distribution of such random variables.
In this direction, we immediately deduce from Theorem \ref{thm:intro1} that for every sequence $\x\in \ell^{\omega_-}_+$ different from the null sequence, $\x\neq {\mathbf 0}\coloneqq(0,0, \ldots)$, the distribution of the random variable $\A_{\x}$
  is also absolutely continuous. More precisely, we have 
   $${\mathcal P}(\A_{\x}\in \d r) = a(\x,r) \d r,$$
where the density $ a(\x,\bullet)$ belongs to ${\mathcal C}^{\infty}_0(\R_+^*)$. In particular, 
$a({\mathbf 1},r)= a(r)$ where ${\mathbf 1}=(1,0,0, \ldots)$. 

\begin{lemma}\label{LB0} The map $\x\mapsto a(\x,\bullet)$ is continuous from $\ell^{\omega_-}_+\backslash \{ {\mathbf 0}\}$
to the space ${\mathcal C}_0(\R_+^*)$ of continuous functions on $(0,\infty)$ vanishing both at $0$ and at $\infty$. 
\end{lemma}
\begin{proof} Write $\phi(\theta)={\mathcal E}_1(\e^{i \theta \A})$ for the characteristic function of $\A$
under ${\mathcal P}_1$, so the characteristic function of $a(\x,\bullet)$ for $\x=(x_1, x_2, \ldots)$ is
$$\phi(\x,\theta) = \prod_{j=1}^{\infty} \phi(x_j^{\omega_-}\theta), \qquad \theta \in \R.$$
Consider a sequence $(\x(n))_{n\geq 1}$ in $\ell^{\omega_-}_+$ that converges to $\x\neq {\mathbf 0}$. Since 
$\A_{\x(n)}$ converges to $\A_{\x}$ in $L^1({\mathcal P})$, $\phi(\x(n),\theta)$ converges pointwise to $\phi(\x,\theta)$.  By Fourier inversion, it suffices to check
that this convergence also holds  in $L^1(\d \theta)$.

Consider an index $j\geq 1$ with $x_j>0$, so for all $n$ sufficiently large, we also have $x_j(n)>x_j/2$. 
Using the bound $|\phi(\x(n),\theta)| \leq |\phi(x_j(n)\theta)|$ and, from the proof of Theorem \ref{thm:intro1}, the fact 
that $\phi(\theta)=O(|\theta|^{-b})$ as $|\theta |\to \infty$ for every $b>0$,   we see that dominated convergence applies, and the proof is complete. 
\end{proof} 

The tail-behavior of $\A_{\x}$ can be deduced from \eqref{E:tailA}, at least in the case when $\omega_+/ \omega_-\leq 2$. Indeed  Lemma A.4 in \cite{MikSamorod} shows that then
\begin{equation}\label{E:tailAx}
{\mathcal P}(\A_{\x}> r) \sim c \left(\sum_{j=1}^{\infty} x_j^{\omega_+}\right) r^{-\omega_+/\omega_-}\qquad \text{as }r\to \infty.
\end{equation}
We  turn our attention to the more delicate question of the  asymptotic behavior of the density $a(\x, \cdot)$. 
When the sequence $\x$ has only finitely many non-zero terms, 
one easily obtains  the following extension of Theorem \ref{thm:intro2}.

\begin{corollary} \label{CexT2} For every $\x\in \ell^{\omega_-}_+\backslash \{ {\mathbf 0}\}$
with ${\mathrm Card}\{j\geq 1: x_j>0\}<\infty$, we have
$$\lim_{r\to \infty} r^{1+\omega_+/\omega_-} a(\x,r)=  c\frac{\omega_+}{\omega_-}\sum_{j=1}^{\infty} x_j^{\omega_+},$$
where $c$ is the constant appearing in \eqref{E:tailA}.
\end{corollary} 
\begin{proof}
 In the case when sequence $\x$ has a single non-zero term, say $\x=(x_1,0,0, \ldots)$ with $x_1>0$, then the claim follows immediately from Theorem \ref{thm:intro2} since then $x_1^{\omega_-}\A_1$ has the density
\begin{equation}\label{E:asympabis}
 x_1^{-\omega_-} a(rx_1^{-\omega_-}) \sim 
 c\frac{\omega_+}{\omega_-} x_1^{\omega_+} r^{-1-\omega_+/\omega_-}, \qquad \text{as }r\to \infty.
 \end{equation}
 
 Let $m\geq 1$ and suppose that the assertion in the statement holds provided that the sequence $\x$ 
 has at most  $m$ non-zero terms. Consider a sequence $\x=(x_1, \ldots, x_m, x_{m+1},0, 0, \ldots)$
 with $x_1\cdots x_{m+1}>0$. Let $\x'=(x_2, \ldots, x_{m+1}, 0, 0, \ldots)$, so 
 $\A_{\x}$ has the same law as $x_{1}^{\omega_-} \A_1+\A_{\x'}$, where the variables $\A_1$  and $\A_{\x'}$ are implicitly assumed to be independent. Thanks to our assumption,  we have then
$$a(\x',r)\sim  c\frac{\omega_+}{\omega_-}\left(\sum_{j=2}^{m+1} x_j^{\omega_+}\right)  r^{-1-\omega_+/\omega_-} ,\qquad \text{as }r\to \infty.$$
Combining this with \eqref{E:asympabis} and an easy estimate on convolutions of densities with heavy tails which can be found for instance Theorem 2.2 in \cite{Luxemburg} (beware however of a misprint, the second integral $\int_0^{\infty}g(t)\d t$ there should be replaced by $\int_0^{\infty}f(t)\d t$), we have 
$$a(\x,r)\sim a(\x',r) + x_1^{-\omega_-} a(rx_1^{-\omega_-})\qquad \text{ as $r\to \infty$,}$$
and this  proves our claim by induction. 
\end{proof}
 
 We conjecture that  Corollary \ref{CexT2} holds even when $\x$ has infinitely many positive terms,  but we have not been able to prove this rigorously (the difficulty lies in interchanging limits). The proof of the forthcoming Corollary \ref{C5} provides some support to this conjecture, as it will be shown that it holds indeed for almost-all  (with respect to distributions of cell systems) $\x\in \ell^{\omega_-}_+$.  
 The following lower-bound is rather easy and will however be  sufficient for our purposes.
 
 \begin{corollary} \label{CexT2'} For every $\x\in \ell^{\omega_-}_+\backslash \{ {\mathbf 0}\}$, we have
$$\liminf_{r\to \infty} r^{1+\omega_+/\omega_-} a(\x,r)\geq  c\frac{\omega_+}{\omega_-}\sum_{j=1}^{\infty} x_j^{\omega_+},$$
where $c$ is the constant appearing in \eqref{E:tailA}.
\end{corollary} 
 
 \begin{remark} At least when $\omega_+/ \omega_-\leq 2$, the inequality of Corollary \ref{CexT2'} is actually an equality, since 
 a strict inequality would contradict \eqref{E:tailAx}. 
 \end{remark}
 
\begin{proof}
For $j\geq 1$, write $\A_{\x(j)}= \A_{\x} - x_j^{\omega_-}\A_j$. Fix  $\varepsilon>0$,
and note that for every $r_0>0$ and $r>2r_0$, the events
$$\Omega_j(r_0,r,\varepsilon)=\left\{\A_{\x}\in[r,r+\varepsilon), \A_{\x(j)}\leq r_0\right\}=
\left\{x_j^{\omega_-}\A_j +\A_{\x(j)}\in[r,r+\varepsilon), \A_{\x(j)}\leq r_0\right\}
$$
are pairwise disjoint for $j\geq 1$.  For any index $j$ with $x_j\neq 0$,
since $\A_j $ and $\A_{\x(j)}$ are independent, since $x_j^{\omega_-}\A_j$ has the density 
$x_j^{-\omega_-}a (x_j^{-\omega_-}\cdot)$, and since $\A_{\x(j)}\leq \A_{\x}$,
there is the lower-bound
\begin{align*}
{\mathcal P}\left(\Omega_j(r_0,r,\varepsilon)\right) &\geq \varepsilon {\mathcal P}(\A_{\x(j)}\leq r_0)
\min \{ x_j^{-\omega_-}a(sx_j^{-\omega_-}): r-r_0\leq s \leq r+\varepsilon\}\\
&\geq \varepsilon {\mathcal P}(\A_{\x}\leq r_0) x_j^{\omega_+} \left({r+\varepsilon}\right)^{-1-\omega_+/\omega_-}
\min \{ a(u)u^{1+\omega_+/\omega_-}: u\geq (r-r_0)x_j^{-\omega_-}\} .
\end{align*}
Taking the sum for all $j$ in the inequality above, dividing by $\varepsilon$ and letting $\varepsilon\to 0+$, we get
$$a(\x,r) \geq  r^{-1-\omega_+/\omega_-}  {\mathcal P}(\A_{\x}\leq r_0)  \sum_{j=1}^{\infty} \left(x_j^{\omega_+}
\min \{ a(u)u^{1+\omega_+/\omega_-}: u\geq (r-r_0)x_j^{-\omega_-}\}\right). 
$$
We conclude from Theorem \ref{thm:intro2} and monotone convergence that for every $r_0>0$, 
$$\liminf_{r\to \infty} r^{1+\omega_+/\omega_-} a(\x,r)\geq  c\frac{\omega_+}{\omega_-}{\mathcal P}(\A_{\x}\leq r_0) \sum_{j=1}^{\infty} x_j^{\omega_+}.$$
Finally letting $r_0\to \infty$ yields 
$$\liminf_{r\to \infty} r^{1+\omega_+/\omega_-} a(\x,r)\geq   c\frac{\omega_+}{\omega_-} \sum_{j=1}^{\infty} x_j^{\omega_+}.$$
\end{proof}

\section{Conditioning on the intrinsic area}
\label{sec:conditioning}

We shall now apply results of the preceding section and first construct a regular version of cell systems conditioned on having a given  intrinsic area. We shall then investigate the asymptotic behavior of these conditional distributions, when the value of the intrinsic area tends to infinity, and when the initial mass of the Eve cell tends to $0$.

\subsection{Conditioning a cell system by probability tilting}
\label{ssec:tilting}

Recall  that ${\mathcal B}(n)$ denotes the point measure of the masses at birth of cells  for the $(n+1)$-th generation, that  ${\mathcal G}(n)=\sigma \left({\mathcal X}_u: |u|\leq n\right)$ stands for the $\sigma$-field generated by the trajectories of cells with generation at most $n$, and that ${\mathcal B}(n)$ is ${\mathcal G}(n)$-measurable.

\begin{theorem}\label{T2} 
Under ${\mathcal P}_1$,  for every $r>0$, the process
$\left( a({\mathcal B}(n),r)\right)_{n\geq 0}$
is a ${\mathcal G}(n)$-martingale with 
$${\mathcal E}_1(a({\mathcal B}(n),r))=a(r)>0.$$ 
\end{theorem}

\begin{proof} We see from the branching property of cell systems, the definition of the intrinsic area and that of the density $a(\x,r)$, that for all $r>0$ and $n\geq 0$:
$$a({\mathcal B}(n),r))= \lim_{\varepsilon\to 0+} 
\varepsilon^{-1} {\mathcal P}_1(\A\in[r,r+\varepsilon)\mid {\mathcal G}(n)) \qquad {\mathcal P}_1\text{-a.s.}
$$
It then follows from Fatou's lemma that
\begin{equation}\label{E:ineqB}
{\mathcal E}_1(a({\mathcal B}(n),r))\leq  a(r).
\end{equation}

A similar argument  using now the conditional version of Fatou's lemma and the tower property of conditional expectation,  shows that for every $r>0$, $\left( a({\mathcal B}(n),r)\right)_{n\in\N}$ is a ${\mathcal G}(n)$-supermartingale under ${\mathcal P}_1$.
Because nonnegative supermartingales with a constant expectation are necessarily martingales, $\left( a({\mathcal B}(n),r)\right)_{n\geq 0}$ is a martingale whenever \eqref{E:ineqB} is actually an equality for all $n\geq 0$.

We next use Tonelli's theorem and write
$$
\int_0^{\infty}{\mathcal E}_1(a({\mathcal B}(n),r))\d r = {\mathcal E}_1\left( \int_0^{\infty}a({\mathcal B}(n),r)\d r \right) 
=1= \int_0^{\infty}a(r)\d r.
$$
Comparing with \eqref{E:ineqB}, we conclude that
${\mathcal E}_1(a({\mathcal B}(n),r))=  a(r) $  for Lebesgue-almost all $r>0$. Therefore our claim is proved except on a set with zero Lebesgue measure, and in particular, except on a nowhere dense subset on $(0,\infty)$. 

To complete the proof, take any $r>0$ and consider a sequence $(r_k)_{k\geq 0}$ of positive real numbers converging to $r$, and such that for each fixed $k$, $\left( a({\mathcal B}(n),r_k)\right)_{n\geq 0}$ is a martingale. 
By continuity of the density $a(\x,\bullet )$ for every sequence $\x\neq {\mathbf 0}$ in $\ell^{\omega_-}_+$, we know that for every $n\geq0$,
\begin{equation} \label{E:limps}
\lim_{k\to \infty} a({\mathcal B}(n),r_k)= a({\mathcal B}(n),r)\qquad\text{${\mathcal P}_1$-a.s.  }
\end{equation} 
We just need to check that the convergence also holds in $L^1({\mathcal P}_1)$, since then 
$${\mathcal E}_1(a({\mathcal B}(n),r)) = \lim_{k\to \infty} {\mathcal E}_1 (a({\mathcal B}(n),r_k)) = \lim_{k\to \infty} a(r_k) = a(r).$$

On the one hand, recall from Theorem \ref{thm:intro1} that $\|a\|_{\infty}\coloneqq \sup_{r\in\R} a(r)<\infty$. For every $x>0$, the density of $\A$ under ${\mathcal P}_x$ is $x^{-\omega_-} a(\bullet x^{-\omega_-})$
and therefore bounded from above by $x^{-\omega_-}\|a\|_{\infty}$. 
It follows from convolution that for every sequence $\x\neq {\mathbf 0}$ in $\ell^{\omega_-}_+$, there is the bound
$$\|a(\x, \cdot)\|_{\infty} \leq\|a\|_{\infty}  (\max_{j\geq 1} x_j)^{-\omega_-}.$$
On the other hand, 
recall that $\left({\mathcal B}(n)\right)_{n\geq 0}$ is a multiplicative branching random walk on $(0,\infty)$,
so that if we denote by $\beta_*(n)$ the location of its largest atom, then 
$${\mathcal E}_1(\beta_*(n)^{-\omega_-}) \leq {\mathcal E}_1(\beta_*(0)^{-\omega_-})^{n+1}.$$
Thanks to \eqref{E:mom-b} (note that the quantity $\gamma_1$ there coincides with 
$\beta_*(0)^{\omega_-}$ here), the right-hand side is finite. This enables us to apply dominated convergence
in \eqref{E:limps} and the proof is complete. 
\end{proof}

Theorem \ref{T2} enables us to construct new probability distributions for cell systems by tilting. Fix $r>0$ and recall from Theorem \ref{thm:intro2} that $a(r)>0$. We define unambiguously 
for any event $ B\in{\mathcal G}(n)$
\begin{equation}\label{E:Doob}
{\mathcal P}_1(B\mid \A=r)= \frac{1}{a(r)} {\mathcal E}_1\left( a({\mathcal B}(n),r) {\mathbf1}_B\right), 
\end{equation}
and by the Daniell-Kolmogorov extension theorem,  this yields a distribution on the space of cell systems which we denote by ${\mathcal P}_1(\bullet \mid \A=r)$.  We now justify the notation, and check that indeed this yields a disintegration of ${
\mathcal P}_1$ with respect to the intrinsic area. 

\begin{corollary}\label{C6}
For every measurable function $f: (0,\infty)\to \R_+$ and every functional $G\geq 0$ of cell systems, there is the identity
$${\mathcal E}_1(G({\mathcal X})f(\A)) = \int_0^{\infty} {\mathcal P}_1(G({\mathcal X}) \mid \A=r) f(r) a(r) \d r.$$
\end{corollary}
\begin{proof} Suppose first that the functional $G$ is $ {\mathcal G}(n)$-measurable for some $n\geq 0$. Then we  write 
$${\mathcal E}_1(G({\mathcal X})f(\A)) = {\mathcal E}_1(G({\mathcal X}){\mathcal E}_1(f(\A) \mid {\mathcal G}(n)) ) = {\mathcal E}_1\left( G({\mathcal X}) \int_0^{\infty} f(r) a({\mathcal B}(n),r)\d r \right),$$
and Tonelli's theorem enables to express the right-hand side in the form 
$$\int_0^{\infty} f(r) {\mathcal E}_1\left( G({\mathcal X})  a({\mathcal B}(n),r) \right) \d r=
 \int_0^{\infty} {\mathcal E}_1(G({\mathcal X}) \mid \A=r) f(r)  a(r) \d r.$$
 Our claim is proved when  $G$ is $ {\mathcal G}(n)$-measurable, and the general case follows from the monotone class theorem. 
\end{proof}

We transfer the preceding results by scaling to the situation where the initial size of the Eve cell is 
arbitrary. Specifically, recall the notation \eqref{E:scalecell} and \eqref{E:scalecellloi}, and note that the intrinsic area of the rescaled cell system ${\mathcal X}^{(b)}$ is $\A^{(b)}=b^{\omega_-} \A$. 
For every $r,x>0$, we then define ${\mathcal P}_x(\bullet \mid \A=r)$ as the law of ${\mathcal X}^{(x)}$ under 
${\mathcal P}_1(\bullet \mid \A=rx^{-\omega_-})$, and readily deduce from \eqref{E:scalecellloi} and Corollary \ref{C6} that the family $\left( {\mathcal P}_x(\bullet \mid \A=r)\right)_{ r>0}$ is indeed a  regular version of the disintegration of ${
\mathcal P}_x$ with respect to the intrinsic area. In this vein, we point at the following extension of \eqref{E:Doob}.

\begin{lemma}\label{L:2B}
For every $r,x>0$, one has for every event $ B\in{\mathcal G}(n)$ 
$${\mathcal P}_x(B\mid \A=r)=  \frac{x^{\omega_-}}{a(rx^{-\omega_-})} {\mathcal E}_x\left( a({\mathcal B}(n),r) {\mathbf1}_B\right)\,. $$
\end{lemma}
\begin{proof} To start with, observe that for every $b>0$ and $\x\in \ell^{\omega_-}_+$, one has $\A_{b\x}=b^{\omega_-}\A_{\x}$ and therefore there is the identity
$$a(b\x,r) = b^{-\omega_-} a(\x,b^{-\omega_-}r)\qquad \text{for all }r>0.$$
For any ${\mathcal G}(n)$-measurable functional $G\geq 0$ , we have
by  \eqref{E:Doob}:
\begin{align*}
{\mathcal E}_x(G({\mathcal X}) \mid \A=r)&={\mathcal E}_1(G({\mathcal X}^{(x)}) \mid \A=rx^{-\omega_-})\\
&= \frac{1}{ a(rx^{-\omega_-})} {\mathcal E}_1\left( G({\mathcal X}^{(x)})a({\mathcal B}(n),rx^{-\omega_-})\right)\\
&= \frac{x^{\omega_-}}{a(rx^{-\omega_-})} {\mathcal E}_1\left( G({\mathcal X}^{(x)})a(x{\mathcal B}(n),r)\right)\\
&= \frac{x^{\omega_-}}{a(rx^{-\omega_-})} {\mathcal E}_x\left( G({\mathcal X})a({\mathcal B}(n),r)\right),
\end{align*}
where we used again  \eqref{E:scalecellloi} at the last line. 
\end{proof}

We conclude this section with another standard observation relating conditioning and rescaling.
\begin{corollary}\label{C4}  Consider the random rescaling  \eqref{E:scalecellloi} for
 $b=\A^{-1/\omega_-}$. The law of the rescaled cell system ${\mathcal X}^{(\A^{-1/\omega_-})}$ under ${\mathcal P}_1$
 is a mixture of the conditional laws $\left( {\mathcal P}_x(\bullet \mid \A=1)\right)_{x>0}$; specifically
 we have for every functional $G\geq 0$ that
 $${\mathcal E}_1\left(G({\mathcal X}^{(\A^{-1/\omega_-})})\right) =\omega_- \int_0^{\infty}  a(x^{-\omega_-}){\mathcal E}_x\left(G({\mathcal X})\mid \A=1\right)
\frac{\d x}{x^{1+\omega_-}}.$$
\end{corollary}
\begin{proof} We first use Corollary \ref{C6} to write 
 $${\mathcal E}_1\left(G({\mathcal X}^{(\A^{-1/\omega_-})})\right) =\int_0^{\infty}  a(r){\mathcal E}_1\left(G({\mathcal X}^{(\A^{-1/\omega_-})})\mid \A=r\right)\d r = \int_0^{\infty}  a(r){\mathcal E}_1\left(G({\mathcal X}^{(r^{-1/\omega_-})})\mid \A=r\right)\d r.$$
Then it suffices to recall that we defined the conditional law ${\mathcal P}_x(\bullet \mid \A=1)$
as that of ${\mathcal X}^{(x)}$ under 
${\mathcal P}_1(\bullet \mid \A=x^{-\omega_-})$ and perform the change of variables $x=r^{-1/\omega_-}$.
\end{proof}

\subsection{Conditioning on  a large given area}
We next derive from the preceding section a first limit theorem for cell systems conditioned on $\A=r\gg 1$.
In this direction, recall from Section 3.1 that ${\mathcal M}^+(n)$ denotes the natural martingale associated to the masses at birth of cells at the $n$-th generation, which has terminal value $0$ ${\mathcal P}_1$-a.s.
\begin{corollary}\label{C5} 
Let $n\geq 0$ and $G\geq 0$ be a functional of cell systems that is ${\mathcal G}(n)$-measurable.
Then
$$\lim_{r\to \infty}  {\mathcal E}_1(G({\mathcal X}) \mid \A=r)=  {\mathcal E}_1(G({\mathcal X}){\mathcal M}^+(n)).$$
\end{corollary}
\begin{proof} Using \eqref{E:Doob}, all that we need to check is that 
$$\lim_{r\to \infty} \frac{a({\mathcal B}(n),r)}{a(r)} = {\mathcal M}^+(n)\qquad \text{in }L^1({\mathcal P}_1).$$
We know already from Theorem \ref{thm:intro2} and Corollary \ref{CexT2'} that
$$\liminf_{r\to \infty} \frac{a({\mathcal B}(n),r)}{a(r)} \geq {\mathcal M}^+(n),$$
and an easy variation of the Riesz-Scheff\'e lemma enables us to conclude. More precisely, we have on the one hand
$$\lim_{r\to \infty} \left( \frac{a({\mathcal B}(n),r)}{a(r)} \wedge {\mathcal M}^+(n)\right) = {\mathcal M}^+(n),$$
where, by Lebesgue's theorem, this convergence holds in $L^1({\mathcal P}_1)$. 
On the other hand, we have also (recall Theorem \ref{T2})
$${\mathcal E}_1(a({\mathcal B}(n),r))= a(r) \quad \text{and} \quad {\mathcal E}_1({\mathcal M}^+(n))=1,$$
and therefore 
$${\mathcal E}_1\left( \left| \frac{a({\mathcal B}(n),r)}{a(r)} - {\mathcal M}^+(n)\right| \right) = 
{\mathcal E}_1\left( \frac{a({\mathcal B}(n),r)}{a(r)} + {\mathcal M}^+(n) \right) - 2 {\mathcal E}_1\left( \frac{a({\mathcal B}(n),r)}{a(r)} \wedge {\mathcal M}^+(n)\right)$$
converges to $0$ as $r\to \infty$. 
\end{proof}
It might be worth to interpret Corollary \ref{C5} in terms of growth-fragmentations $\X$ rather than cell systems ${\mathcal X}$. Specifically, the probability-tilting of ${\mathcal P}_1$ based on the martingale $({\mathcal M}^+(n))_{n\geq 0}$ can be viewed as  conditioning on indefinite growth. The distribution of the growth-fragmentation $\X$ under the tilted probability is $\Q^+_1$, that is that of the growth-fragmentation associated to a cell system with law ${\mathcal Q}^+_1$. See Section 4 in \cite{BBCK} for details.
Roughly speaking, this shows that conditioning a self-similar growth-fragmentation with law $\P_1$ on having a large intrinsic area $\A=r\gg 1$ amounts asymptotically  to conditioning this growth-fragmentation on having indefinite growth, and this merely consists of replacing the dynamics of the Eve cell by those of a SSMP with characteristics $(\Phi^+,\alpha)$ without modifying those of cells at generation $n\geq 1$. 

Corollary \ref{C5} immediately extends to the conditional laws ${\mathcal P}_x(\bullet \mid \A=r)$ for any $x>0$
by scaling. In this direction, it is interesting to recall from Section 4.2 of \cite{BBCK} that when $\alpha <0$, $x=0$ is an entrance point for the growth-fragmentation conditioned on having indefinite growth. This suggests that  conditioning on the intrinsic area may then produce a non-degenerate process when the growth-fragmentation starts from $0$; this question is addressed in the next section.

\subsection{Conditioning the canonical measure on its intrinsic area}
\label{ssec:condarea}

In this final section, we assume that $\alpha<0$. Our purpose is to  construct a process which can be thought of as the original growth-fragmentation started from $0$ and conditioned to have an intrinsic area $\A=r>0$. 

To start with, recall that even though  in general  for any $x>0$, the distributions of cell systems  ${\mathcal P}_x$ and ${\mathcal Q}^-_x$
 are mutually singular, the laws of the growth-fragmentation $\X$ that they induce are actually equivalent. Specifically, there is the identity
$\d \Q_x^{-}= x^{-\omega_-}\A \d \P_x$ (see Section 4.3 in \cite{BBCK}), and  this   implies that the conditional distributions of the growth-fragmentation given its intrinsic area are the same for $\P_x$ and $\Q^-_x$,
i.e. $\Q^-_x(\bullet \mid \A=r)=\P_x(\bullet \mid \A=r)$ for all $r>0$. Here, it will be more convenient for us to work under the area-biased distribution $\Q^{-}_x$.

Roughly speaking, we would like  to condition the growth-fragmentation on having intrinsic area $\A=r>0$ when the Eve cell has initial mass $0$, i.e. to take $x=0$ in what precedes.  An obvious obstacle is that the probability measure ${\mathcal Q}^-_0$ on cell systems is clearly degenerate, in the sense that no individual has ever a positive mass.
This seems to impede making any sense to such a conditioning; nonetheless this obstruction can be circumvented 
by applying general results of Rivero \cite{Riv} on the existence of pseudo-excursion measures for SSMP.
Indeed, one can define a non-degenerate $\sigma$-finite  measure
under which the Eve cell starts from $0$ and has the transitions of a SSMP with characteristics $(\Phi^-,\alpha)$ (we stress that this would fail if we did not assume that $\alpha <0$).
More precisely, recall  \eqref{E:relphi}; the process $(X^{\omega_{\Delta}}(t))_{ t>0}$ is a $Q^-_x$-martingale for every $x>0$, and there is the relation of  local  absolutely continuity between the SSMP with characteristics $(\Phi^-,\alpha)$ and $(\Phi^+,\alpha)$:
$$x^{-\omega_{\Delta}} Q_x^-(X^{\omega_{\Delta}}(t){\mathbf 1}_B, \zeta>t) = Q^+_x(B), \qquad\text{for any event }B\in\f_t$$
(recall that $({\mathcal F}_t)_{t\geq 0}$ denotes the canonical filtration on the space of trajectories). 
In turn, $(X^{-\omega_{\Delta}}(t))_{t>0}$ is a $Q^+_x$-supermartingale; this enables us to introduce the $\sigma$-finite measure $n^-_0$ on the space of trajectories given by
\begin{equation}\label{defn-}
n^-_0(B, \zeta>t) \coloneqq Q^+_0(X^{-\omega_{\Delta}}(t){\mathbf 1}_B) \qquad\text{for any event }B\in\f_t, 
\end{equation}
where in the right-hand side, $Q^{+}_0=\lim_{x\to 0+} Q^{+}_x$ is a well-defined non-degenerate law
(see, for instance, \cite{BY}). We may thus think of the pseudo excursion measure $n^-_0$ as  the weak  limit of $ x^{-\omega_{\Delta}} Q_x^-$ as $x\to 0+$. 

In this setting, we underline a useful scaling relation: in the notation \eqref{E:scalenot},\begin{equation}\label{scale2}
\hbox{ for every $b>0$, the processes $X$ and $X^{(b)}$
have the same law under $Q^+_0$.}
\end{equation}
In turn, this yields \begin{equation}\label{scale3}
\hbox{ for every $b>0$, the distribution of $X^{(b)}$ under $n^-_0$ is $b^{\omega_{\Delta}}n^-_0$
.}
\end{equation}

We can now endow cell systems with a  $\sigma$-finite measure ${\mathcal N}^-_0$ under which 
the trajectory of the Eve cell,  ${\mathcal X}_{\varnothing}$, is distributed according to $n^-_0$, whereas the cells at generation $1,2, \ldots$ follow
the dynamics of the SSMP with characteristics $(\Psi,\alpha)$. We call ${\mathcal N}^-_0$ the canonical measure  and note  from \eqref{defn-} and the Crump-Mode-Jagers type branching structure of cell systems, that 
on any finite time horizon, the canonical measure ${\mathcal N}^-_0$ and the law ${\mathcal Q}^+_0$ are
related via
$${\mathcal N}^-_0({C}, \zeta_{\varnothing}>t) = {\mathcal E}^+_0\left({\mathcal X}_{\varnothing}(t)^{-\omega_{\Delta}} {\mathbf 1}_{C}\right),$$
for any event ${C}$ which is measurable with respect to the trajectories of the cell system observed up to 
the (absolute) time $t$ only. In particular, since the map ${\mathcal X}\mapsto \X$ turning a cell system into a growth-fragmentation is well-defined ${\mathcal Q}^+_0$-a.s. (see Lemma 4.3 in \cite{BBCK}), the same holds under the canonical measure ${\mathcal N}^-_0$.

Plainly, the notion of the intrinsic area $\A$ still makes sense under the canonical measure, and more precisely,  the smoothing transform reads as follows. 
If 
 $(\gamma^-_i)_{i\geq 1}$ stands for some enumeration of
$\{\!\! \{|\Delta{X}(t)|^{\omega_-}: 0<t<\zeta\}\!\!\}$ under $n^-_0$
and
 $(\A_i)_{i\geq 1}$ denotes as usual a sequence of i.i.d.~copies of $\A$ under ${\mathcal P}_1$
 which is further  independent of $X$, then $\sum_{i=1}^{\infty} \gamma^-_i\A_i$
 has the distribution of  $\A$ under ${\mathcal N}^-_0$.

\begin{lemma} \label{L1} Assume $\alpha <0$. 
The tail-distribution
 of the intrinsic area $\A$ under the canonical measure 
 is given by 
  $${\mathcal N}^-_0(\A >r)= c r^{-\omega_{\Delta}/\omega_-}, \qquad r>0,$$
where $c$ is the constant appearing in \eqref{E:tailA}. 
\end{lemma}
\begin{proof} Write $h_{\varepsilon}=\inf\{t>0: X(t)=\varepsilon\}$ for the first hitting time of 
$\varepsilon >0$. It follows from  \eqref{defn-} (or more precisely its easy extension to stopping times) that $n^-_0(h_{\varepsilon}<\infty)=\varepsilon^{-\omega_{\Delta}}$. This enables us to define a conditional probability measure $n^-_0(\bullet \mid h_{\varepsilon}<\infty)= \varepsilon^{\omega_{\Delta}}n^-_0(\bullet, h_{\varepsilon}<\infty)$ under which the shifted process $(X(h_{\varepsilon}+t))_{t\geq 0}$ has the law $Q^-_{\varepsilon}$. 

Let us restrict the multiset in the smoothing transform under ${\mathcal N}^-_0$ that has been described just  before the statement to jumps that occur after time $h_{\epsilon}$ only, i.e. $\{\!\! \{|\Delta{X}(t)|^{\omega_-}: h_{\varepsilon}<t<\zeta\}\!\!\}$. This yields a variable 
denoted by $\A^{(\varepsilon)}$ with, by convention, $\A^{(\varepsilon)}=0$ on the event $\{h_{\epsilon}=\infty\}$. We see from the first paragraph of this proof that there is the identity
$${\mathcal N}^-_0(\A^{(\varepsilon)}>r)= \varepsilon^{-\omega_{\Delta}} 
{\mathcal Q}^-_{\varepsilon}(\A>r),$$
and then we conclude from Lemma \ref{LA} and the scaling property, that
$$\lim_{\varepsilon \to 0+}{\mathcal N}^-_0(\A^{(\varepsilon)}>r) = c r^{-\omega_{\Delta}/\omega_-}.$$
This entails our claim by monotone convergence, since obviously $\A^{(\varepsilon)}$ increases to $\A$ as $\varepsilon$ decreases to $0$, ${\mathcal N}^-_0$-a.s. 
\end{proof}

We now conclude this work by adapting Corollary \ref{C4} and constructing a regular version of the canonical measure ${\mathcal N}^-_0$ conditioned on having a given area $\A$. Specifically, for every $r>0$, we define 
the probability measure ${\mathcal N}^-_0(\bullet \mid  \A  =r)$ on the space of cell systems
as the image of the probability measure $ {\mathcal N}^-_0(\bullet \mid \A  >1)= c^{-1} {\mathcal N}^-_0(\bullet ,  \A  >1)$ 
by the rescaling map ${\mathcal X} \mapsto  {\mathcal X}^{(\beta)}$ with $\beta= (r/\A)^{1/\omega_-}$. 
Plainly, the intrinsic area computed for the rescaled cell system ${\mathcal X}^{(\beta)}$ equals $r$, and we can now justify our notation:

\begin{proposition} \label{P2} Assume $\alpha <0$. 
 We have
$${\mathcal N}^-_0\left(\bullet\right)= c \frac{\omega_{\Delta}}{\omega_-} \int_0^{\infty}
{\mathcal N}^-_0(\bullet \mid \A =r)  r^{-1+\omega_{\Delta}/\omega_-} \d r.$$
\end{proposition}

\begin{proof} The canonical measure  inherits self-similarity from \eqref{scale3}; namely, for every $b>0$, there is the identity 
$$
\hbox{ the distribution of ${\mathcal X}^{(b)}$ under ${\mathcal N}^-_0$ is $b^{\omega_{\Delta}} {\mathcal N}^-_0$.
}
$$
This readily entails that for any $R>0$,  the image of the conditional probability measure $ {\mathcal N}^-_0(\bullet \mid \A  >R)= c^{-1} R^{\omega_{\Delta}/\omega_-}{\mathcal N}^-_0(\bullet ,  \A  >R)$ 
by the rescaling map ${\mathcal X} \mapsto  {\mathcal X}^{(\beta)}$ with $\beta= (r/\A)^{1/\omega_-}$,
is also given by ${\mathcal N}^-_0(\bullet \mid  \A  =r)$.  Hence, under $ {\mathcal N}^-_0(\bullet \mid \A  >1)$ as well as under $ {\mathcal N}^-_0(\bullet \mid \A  >R)$,  ${\mathcal X}^{(\beta)}$ is independent of $\A$. We thus have
$${\mathcal N}^-_0\left(\bullet, \A>R\right)= c \frac{\omega_{\Delta}}{\omega_-} \int_R^{\infty}
{\mathcal N}^-_0(\bullet \mid \A =r)  r^{-1+\omega_{\Delta}/\omega_-} \d r.$$
Since $R$ may be chosen as small as we wish, the proof is complete. 
\end{proof}

\bibliographystyle{abbrv}
%\bibliography{Biblaire}

\end{document}